\PassOptionsToPackage{unicode}{hyperref}
\PassOptionsToPackage{naturalnames}{hyperref}
\documentclass[a4paper]{amsart}
\usepackage{latexsym,bm,stmaryrd}
\usepackage[a4paper,hmargin=26mm,vmargin=28mm]{geometry}
\usepackage{latexsym,bm}
\usepackage[svgnames]{xcolor}
\usepackage{xparse}
\usepackage{listings}
\usepackage{tikz}
\usetikzlibrary{matrix}
\usepackage[enableskew,vcentermath]{youngtab}
\usepackage[boxsize=1.25em, centerboxes]{ytableau}

\usepackage[utf8]{inputenc}
\usepackage[english]{babel}
\usepackage[T1]{fontenc}
\usepackage{amsmath, amsfonts, amssymb, amsthm, mathrsfs, pb-diagram}
\usepackage{lmodern}
\usepackage{fullpage}
\usepackage{microtype}
\usepackage{enumerate}
\usepackage{mathtools}
\usepackage{enumitem}
\setlist[enumerate]{label=\alph*\upshape), nolistsep}

\usepackage{todonotes}

\NewDocumentCommand\Cmd{ sm }{\textsf{\textbackslash #2}\IfBooleanT{#1}{$\{\ldots\}$}}

\usepackage{etoolbox}
\usepackage{aliascnt}

\usepackage[pagebackref=false]{hyperref}

\newcommand\enumref[2]{\hyperref[#2]{\autoref*{#1}(\autoref*{#2})}}

\usepackage[misc]{ifsym}

\def\NewTheorem#1{%
  \newaliascnt{#1}{equation}%
  \newtheorem{#1}[#1]{#1}%
  \aliascntresetthe{#1}%
  \expandafter\def\csname #1autorefname\endcsname{#1}%
}

\newcommand\blam{{\boldsymbol\lambda}}
\newcommand\brho{{\boldsymbol\rho}}
\newcommand\bmu{{\boldsymbol\mu}}
\newcommand\bnu{{\boldsymbol\nu}}

\let\<=\langle
\let\>=\rangle
\def\({\big(}
\def\){\big)}

\def\Z{\mathbb{Z}}
\def\Q{\mathbb{Q}}

\def\N{\mathbb{N}}
\def\P{\mathscr{P}}
\def\mL{\mathcal{L}}
\def\mC{\mathscr{C}}
\def\mT{\mathcal{T}_0}

\def\p{\mathfrak{p}}
\def\q{\mathfrak{q}}
\def\a{\mathfrak{a}}
\def\b{\mathfrak{b}}

\def\t{\mathfrak{t}}
\def\s{\mathfrak{s}}
\def\u{\mathfrak{u}}
\def\v{\mathfrak{v}}

\def\fm{\mathfrak{m}}
\def\fn{\mathfrak{n}}
\def\bS{\mathbf{S}}
\def\bT{\mathbf{T}}
\def\bU{\mathbf{U}}
\def\bV{\mathbf{V}}

\newcommand\balpha{{\boldsymbol\alpha}}
\newcommand\bbeta{{\boldsymbol\beta}}

\newcommand{\Parts}[1][m]{\P_{#1}}

\def\bQ{\mathbf{Q}}
\def\tlam{\t^{\blam}}
\def\O{\mathcal{O}}
\def\hO{\hat{\O}}

\def\K{\mathscr{K}}
\def\lam{\lambda}

\def\Sym{\mathfrak{S}}

\DeclareMathOperator\Gedom{\,{\underline{\kern-.1ex{\blacktriangleright}\kern-0.1ex}}\,}
\def\fm{\mathfrak{m}}
\def\fn{\mathfrak{n}}

\def\mfg{\mathfrak{g}}

\newcommand{\HH}{\mathscr{H}}

\newcommand{\equal}{\overset{\tau}{=}}

\DeclareMathOperator\cont{cont}

\DeclareMathOperator\Hom{Hom}
\DeclareMathOperator\End{End}

\DeclareMathOperator\Tr{Tr}

\DeclareMathOperator\rank{rank}

\DeclareMathOperator\res{res}

\DeclareMathOperator\Std{Std}
\DeclareMathOperator\Shape{Shape}
\DeclareMathOperator\tr{tr}

\DeclareMathOperator\MM{MM}

\DeclareMathOperator\BK{BK}

\title[Trace forms on Hecke algebras and cocenters of cyclotomic Schur algebras]
{Trace forms on the cyclotomic Hecke algebras and cocenters of the cyclotmic Schur algebras}

\author{Zhekun He}\address{School of Mathematical and Statistics\\
  Beijing Institute of Technology\\
  Beijing, 100081, P.R. China}
\email{hzk3002@163.com}

\author{Jun Hu}\address{MIIT Key Laboratory of Mathematical Theory and Computation in Information Security, School of Mathematical and Statistics\\
  Beijing Institute of Technology\\
  Beijing, 100081, P.R. China}
\email{junhu404@bit.edu.cn}

\author{Huang Lin}\address{School of Mathematical and Statistics\\
  Beijing Institute of Technology\\
  Beijing, 100081, P.R. China}
\email{lhmazczy@163.com}

\numberwithin{equation}{section}
\newtheorem{prop}[equation]{Proposition}
\newtheorem{thm}[equation]{Theorem}
\newtheorem{cor}[equation]{Corollary}

\newtheorem{ques}[equation]{Question}

\newtheorem{lem}[equation]{Lemma}

\newtheorem{exmp}[equation]{Example}

\theoremstyle{definition}
\newtheorem{dfn}[equation]{Definition}

\theoremstyle{remark}
\newtheorem{rem}[equation]{Remark}

\begin{document}

\begin{abstract} We define a unified trace form $\tau$ on the cyclotomic Hecke algebras $\mathscr{H}_{n,K}$ of type $A$, which generalize both Malle-Mathas' trace form on the non-degenerate version (with Hecke parameter $\xi\neq 1$) and Brundan-Kleshchev's trace form on the degenerate version. We use seminormal basis theory to construct an explicit pair of dual bases for $\HH_{n,K}$ with respect to the form. We also construct an explicit basis for the cocenter (i.e., the $0$th Hochschild homology) of the corresponding cyclotomic Schur algebra, which shows that the cocenter has dimension independent of the ground field $K$, the Hecke parameter $\xi$ and the cyclotomic parameters $Q_1,\cdots,Q_\ell$.
\end{abstract}


\maketitle
\setcounter{tocdepth}{1}

\section{Introduction}

Throughout this paper, let $n\in\N$, $\ell\in\Z^{\geq 1}$, $R$ be an integral domain and $\xi\in R^\times, Q_1,\cdots,Q_\ell\in R$. The {\bf non-degenerate cyclotomic Hecke algebras} $\mathcal{H}_{\ell,n}$ of type $G(\ell,1,n)$ were first introduced in \cite{AK}, \cite{BM:cyc} and \cite{C}, and they play important roles in the modular representation theory of finite groups of Lie type over fields of non-defining characteristic. By definition, $\mathcal{H}_{\ell,n}$ is the unital associative $R$-algebra with generators $T_{0}, T_{1}, \cdots, T_{n-1}$ that are subject to the following relations:
$$
\begin{aligned}
&\left(T_{0}-Q_{1}\right) \cdots\left(T_{0}-Q_{\ell}\right)=0, \quad T_{0} T_{1} T_{0} T_{1}=T_{1} T_{0} T_{1} T_{0}; \\
&\left(T_{i}-\xi \right)\left(T_{i}+1\right)=0, \quad \forall\, 1 \leq i \leq n-1; \\
&T_{i} T_{j}=T_{j} T_{i},\quad  \forall\, 1 \leq i<j-1<n-1; \\
&T_{i} T_{i+1} T_{i}=T_{i+1} T_{i} T_{i+1}, \quad \forall\, 1 \leq i<n-1.
\end{aligned}
$$
We call $\xi$ the Hecke parameter and  $Q_{1}, \cdots, Q_{\ell}$ the cyclotomic parameters of $\mathcal{H}_{\ell,n}$. The {\bf Jucys-Murphy elements} of  $\mathcal{H}_{\ell,n}$ are defined as:
\begin{equation}
\mL_{m}:=\xi^{1-m} T_{m-1} \cdots T_{1} T_{0} T_{1} \cdots T_{m-1},\quad m=1,2, \cdots, n.
\end{equation}
These elements commutes with each other. By \cite{AK}, the set
\begin{equation}\label{basis120}
\bigl\{\mL_1^{c_1}\mL_2^{c_2}\cdots \mL_n^{c_n}T_w\bigm|w\in\Sym_n, 0\leq c_i<\ell,\forall\,1\leq i\leq n\bigr\}
\end{equation}
form an $R$-basis of $\mathcal{H}_{\ell,n}$. By the work of \cite{BM} and \cite{MM}, there is a symmetrizing form $\tau^{\MM}$ on $\mathcal{H}_{\ell ,n}$ which makes it into a symmetric algebra over $R$. More precisely, for any $w\in \Sym_n$ and integers $0\leq c_1,c_2,\cdots,c_n<\ell$,
$$
\tau^{\MM}(\mL_1^{c_1}\cdots \mL_n^{c_n}T_w):=
\begin{cases}
1, &\text{if $w=1$ and $c_1=\cdots=c_n=0$;}\\
0, &\text{otherwise.}
\end{cases}
$$

The {\bf degenerate cyclotomic Hecke algebras} $H_{\ell,n}$ of type $G(\ell,1,n)$ (\cite{AMR}, \cite{BK08}) is the unital associative $R$-algebra with generators $s_{1},\cdots,s_{n-1}, L_{1}, \cdots, L_{n}$ that are subject to the following relations:
$$
\begin{aligned}
&\left(L_{1}-Q_{1}\right) \cdots\left(L_{1}-Q_{\ell}\right)=0, \quad s_{i}^{2}=1, \quad \forall\, 1 \leq i \leq n-1; \\
&s_{i} s_{j}=s_{j} s_{i}, \forall\, 1 \leq i<j-1<n-1; \\
&s_{i} s_{i+1} s_{i}=s_{i+1} s_{i} s_{i+1},\quad  \forall\, 1 \leq i<n-1; \\
&L_{i} L_{k}=L_{k} L_{i},\,\,\, s_{i} L_{t}=L_{t} s_{i}, \,\,\forall\,1 \leq i<n, 1 \leq k, t \leq n, t \neq i, i+1; \\
&L_{i+1}=s_{i} L_{i} s_{i}+s_{i}, \forall\,1 \leq i<n .
\end{aligned}
$$
The elements $L_{1}, \cdots, L_{n}$ are called the {\bf Jucys-Murphy elements} of $H_{\ell,n}$. We call $Q_{1}, \cdots, Q_{\ell}$ the cyclotomic parameters of $H_{\ell,n}$. By \cite{BK08}, the elements in the set
\begin{equation}\label{basis10}
\bigl\{L_1^{c_1}L_2^{c_2}\cdots L_n^{c_n}w\bigm|w\in\Sym_n, 0\leq c_i<\ell,\forall\,1\leq i\leq n\bigr\}
\end{equation}
form an $R$-basis of $H_{\ell ,n}$.

When $\ell=1$, the degenerate cyclotomic Hecke algebra of type $G(1,1,n)$ is nothing but the group algebra of the symmetric group $\Sym_n$, while the non-degenerate cyclotomic Hecke algebra of type $G(1,1,n)$ is the Iwahori-Hecke algebra of $\Sym_n$. So specialization $\xi$ to $1$, the non-degenerate cyclotomic Hecke algebra of type $G(1,1,n)$ becomes the degenerate cyclotomic Hecke algebra of type $G(1,1,n)$. However, when $\ell > 1$, the previous presentation of the non-degenerate cyclotomic Hecke algebra $\mathcal{H}_{\ell,n}$ does {\it NOT} automatically specialize to the presentation of the degenerate cyclotomic Hecke algebra $H_{\ell,n}$ when specializing $\xi$ to $1$.

Brundan and Kleshchev introduced in \cite{BK08} a symmetrizing form $\tau^{\BK}$ on $H_{\ell,n}$ which makes it into a symmetric algebra over $R$. More precisely,
for any $w\in \Sym_n$ and integers $0\leq c_1,c_2,\cdots,c_n<\ell$, $$
\tau^{\BK}(L_1^{c_1}\cdots L_n^{c_n}w):=
\begin{cases}
1, &\text{if $w=1$ and $c_1=\cdots=c_n=\ell-1$, }\\
0, &\text{otherwise.}
\end{cases}
$$

In \cite[Definition 2.1]{HuMathas:SeminormalQuiver}, Mathas and the second author of this paper have given an unified presentation for the cyclotomic Hecke algebra $\HH_{n,R}$ of type $A$ (which covers both the non-degenerate version and the degenerate version), see Definition \ref{TypeA}. It is natural to ask the following question:

\begin{ques}
Can we give a unified definition of symmetrizing form $\tau$ on $\HH_{n,R}$ which generalizes both $\tau^{\MM}$ and $\tau^{\BK}$?
\end{ques}

The cyclotomic Hecke algebra $\HH_{\ell,n}$ contains the Iwahori-Hecke algebra of type $A_{n-1}$ and $B_n$ as special cases (i.e., $\ell=1$ and $\ell=2$ cases). In the case of types $A_{n-1}$ and $B_n$, we have the Kazhdan-Lusztig bases $\{C'_w|w\in W\}$ (\cite{KL}) and the dual Kazhdan-Lusztig bases $\{D'_w|w\in W\}$ (\cite{Lu0}), where $W\in\{W(A_{n-1}, W(B_n)\}$ is the Weyl group. These two bases are actually a pair of dual bases with respect to the canonical symmetrizing form on the corresponding Iwahori-Hecke algebras. In general, one can ask the following question:

\begin{ques}
Can we construct an explicit pair of dual bases for the cyclotomic Hecke algebra $\HH_{n,R}$ of type $A$ with respect to its symmetrizing form?
\end{ques}

In this paper we shall give affirmative answers to the above two questions. The symmetrizing form on the cyclotomic Hecke algebra $\HH_{n,R}$ implies that there is an $R$-linear isomorphism between the center $Z(\HH_{n,R})$ and the dual of the cocenter $\Tr(\HH_{n,R}):=\HH_{n,R}/[\HH_{n,R},\HH_{n,R}]$ of $\HH_{n,R}$. By the Schur-Weyl duality between the cyclotomic Hecke algebra $\HH_{n,R}$ and the corresponding cyclotomic Schur algebra $S_{n,R}$, one can deduce that there is a natural $R$-algebra isomorphism between the center $Z(\HH_{n,R})$ and the center  $Z(S_{n,R})$ of the cyclotomic Schur algebra $S_{n,R}$, see Corollary \ref{keycor22}. One of the folklore conjecture about $Z(\HH_{n,R})$ says that the center $Z(\HH_{n,R})$ has dimension independent of its ground field as well as it Hecke parameters and cyclotomic parameters. As the third main result of this paper we shall prove that the cocenter of the cyclotomic Schur algebra $S_{n,R}$ has dimension independent of its ground field as well as its Hecke parameters and cyclotomic parameters.\medskip

The paper is organized as follows. In Section 2 we shall recall some basic definitions, properties and combinatorial notations for the cyclotomic Hecke algebras of type $A$ and the corresponding cyclotomic Schur algebras. In Section 3 we study the symmetrizing form on the unified presentation of the cyclotomic Hecke algebras $\HH_{n,R}$ of type $A$. Our first main result Theorem \ref{mainthm11} gives a unified definition (Definition \ref{trace}) of the symmetrizing form $\tau_R$ on $\HH_{n,R}$. We use the seminormal basis theory of $\HH_{n,\K}$ to give a new proof that $\HH_{n,R}$ is a symmetric algebra. As an application, we show in Proposition \ref{keylem22} that the two quasi-idempotents $z_\blam T_{w_{\blam'}}, z_\bmu T_{w_{\bmu'}}$ are $K$-linearly dependent in the cocenter $\Tr(\HH_{n,K})$ if and only if $\blam$ and $\bmu$ are in the same block of $\HH_{n,K}$. Our second main result Theorem \ref{mainthm22} of this paper is given in Section 4, where we use the unified symmetrizing form obtained in Theorem \ref{mainthm11} and the distinguished cellular basis of $\HH_{n,K}$ introduced in \cite{HuMathas:SeminormalQuiver} to construct a pair of the dual bases of $\HH_{n,K}$. In Section 5 we use the seminormal basis theory for the cyclotomic Schur algebras $S_{n,K}$ to study the cocenter (i.e., the $0$th Hochschild homology) $\Tr(S_{n,K})=S_{n,K}/[S_{n,K},S_{n,K}]$  of $S_{n,K}$. Our third main result Theorem \ref{mainthm33} of this paper gives an explicit basis for the cocenter $\Tr(S_{n,K})$ of the cyclotomic Schur algebra $S_{n,K}$. The argument of the proof is quite general and actually applies equally well to any quasi-hereditary algebra which has s seminormal basis theory (e.g., Brauer algebras, BMW algebras). In particular, this implies the cocenter has dimension independent of the ground field $K$, the Hecke parameter $\xi$ and the cyclotomic parameters $Q_1,\cdots,Q_\ell$, and it is stable under base change. As an application, we obtain in Corollary \ref{keycor3} an explicit basis for the commutator subspace $[S_{n,K},S_{n,K}]$, and show in Corollary \ref{maincor1} that the commutator subspace is stable under base change too.

\bigskip\bigskip
\centerline{Acknowledgements}
\bigskip

The research was supported by the National Natural Science Foundation of China (No. 12171029).
\bigskip

\section{Preliminary}

In this section we shall introduce some notations, recall some basic definitions, properties and results about the cyclotomic Hecke algebra of type $A$ and the corresponding cyclotomic Schur algebras.

As before, let $R$ be an integral domain and $\xi\in R^\times, Q_1,\cdots,Q_\ell\in R$.

\begin{dfn}[{\cite{HuMathas:SeminormalQuiver}}]\label{TypeA} The cyclotomic Hecke algebra of type $A$ with Hecke parameter $\xi$ and cyclotomic parameters $Q_{1}, \cdots, Q_{\ell}$ is the unital associative $R$-algebra $\HH_{n}=\HH_{n}(\xi, Q_{1}, \cdots, Q_{\ell})$ with generators $L_{1}, \cdots, L_{n}, T_{1}, \cdots, T_{n-1}$ that are subject to the following relations:
    $$
    \begin{array}{l}
    {\prod_{l=1}^{\ell}\left(L_{1}-Q_{l}\right)=0},\quad  (T_r+1)(T_r-\xi)=0, \quad \text{if $1\leq r<n$},\\
    {L_{r} L_{t}=L_{t} L_{r}, \quad T_{r} T_{s}=T_{s} T_{r} \quad \text { if }|r-s|>1, 1\leq r,t\leq n}, \\
    {T_{s} T_{s+1} T_{s}=T_{s+1} T_{s} T_{s+1},  \quad \text{if $1\leq s<n-1$}},\\
    T_{r} L_{t}=L_{t} T_{r}, \quad \text{if $t \neq r, r+1$}, \\
    {L_{r+1}\left(T_{r}-\xi+1\right)=T_{r} L_{r}+1}, \quad \text{if $1\leq r<n$}.
    \end{array}
    $$
\end{dfn}

\begin{rem}\label{rem.L'}
If $\xi=1$ then $\HH_{n}$ is the degenerate cyclotomic Hecke algebra with the standard presentation given in the introduction; while if $\xi\neq 1$ then $\HH_{n}$ is isomorphic to the non-degenerate cyclotomic Hecke algebra (with cyclotomic parameters $(Q_1(\xi-1)+1,\cdots,Q_\ell(\xi-1)+1$) by setting $\mL_k:=(\xi-1)L_k+1$ for $1\leq k\leq n$ and $T_0:=\mL_1$. Sometimes we use the notation $\HH_{n,R}$ instead of $\HH_{n}$ in order to emphasize the ground ring $R$.
\end{rem}

\begin{lem}[{\cite{AK}, \cite{BK08}}]\label{stdBasis}
The elements in the following set
\begin{equation}\label{basis1}
\bigl\{L_1^{c_1}L_2^{c_2}\cdots L_n^{c_n}T_w\bigm|w\in\Sym_n, 0\leq c_i<\ell,\forall\,1\leq i\leq n\bigr\}
\end{equation}
form an $R$-basis of $\HH_{n,R}$. In the non-degenerate case, the following set
\begin{equation}\label{basis12}
\bigl\{\mL_1^{c_1}\mL_2^{c_2}\cdots \mL_n^{c_n}T_w\bigm|w\in\Sym_n, 0\leq c_i<\ell,\forall\,1\leq i\leq n\bigr\}
\end{equation}
form an $R$-basis of $\mathcal{H}_{\ell,n}$.
\end{lem}

\begin{lem}[\text{\rm \cite{AK}, \cite{AMR}}]\label{ss} Let $R=K$ be a field. The  cyclotomic Hecke algebra $\HH_{n,K}$ is semisimple if and only if $$
\Bigl(\prod_{k=1}^{n}(1+\xi+\cdots+\xi^{k-1})\Bigr)\Bigl(\prod_{\substack{1\leq l<r\leq\ell\\ -n<k<n}}\bigl(Q_r-\xi^kQ_l-(1+\xi+\cdots+\xi^{k-1})\bigr)\Bigr)\in K^\times .
$$
In that case, it is split semisimple.
\end{lem}

Let $d\in\N$. A composition of $d$ is a finite sequence $\alpha=(\alpha_1,\alpha_2,\cdots)$ of non-negative integers which sums to $d$, we write $|\alpha| = \sum_{k\geq 1} \alpha_k=d$.
A multicomposition of $d$ is an ordered $\ell$-tuple $\blam=(\lam^{(1)},\cdots,\lam^{(\ell)})$ of composition $\lam^{(k)}$ such that $\sum_{k=1}^{\ell}|\lam^{(k)}|=d$.
A partition of $d$ is a composition $\lambda=(\lambda_1,\lambda_2,...)$ of $d$ such that $\lam_1\geq\lam_2\geq\cdots$.
A multipartition of $d$ is a multicomposition $\bm{\lambda}=\left(\lambda^{(1)}, \cdots, \lambda^{(\ell)}\right)$ of $d$ such that each $\lam^{(k)}$ is a partition.
Given a composition $\lambda = (\lambda_1 ,\lambda_2 \ldots )$ of $d$, we define its conjugate $\lambda' = (\lambda'_1 , \lambda'_2 ,\ldots)$ by $\lambda'_k = \# \{ j\geq 1 \mid \lambda_j \geq k \}$, which is a partition of $d$.
Given a multicomposition $\bm{\lambda}=\left(\lambda^{(1)}, \cdots, \lambda^{(\ell)}\right)$, we called the multipartition $\bm{\lambda}'=\left({\lambda}^{(\ell)'}, \cdots, \lambda^{(1)'}\right)$ the conjugate of $\bm{\lambda}$.

We identify the multipartition $\bm{\lambda}$ with its multi-Young diagram that is the set of boxes
$$[{\bm{\lambda}}]=\left\{(l, r, c) \mid 1 \leq c \leq \lambda_{r}^{(l)}, 1 \leq l \leq \ell\right\}.$$
For example, if $\bm{\lambda}=\left((3,2,1),(1,1),(3,2) \right)$ then
\begin{equation*}
[\bm{\lambda}] = \Bigg( \ydiagram{3,2,1},~ \ydiagram{1,1},~ \ydiagram{3,2} \Bigg).
\end{equation*}
We use $\P_{\ell,n}$ to denote the set of multipartitions of $n$ with $\ell$ components. Then $\P_{\ell,n}$ becomes a poset ordered by dominance ``$\unrhd$'', where $\blam\unrhd\bmu$ if and only if
$$
\sum_{k=1}^{l-1}\left|\lambda^{(k)}\right|+\sum_{j=1}^{i} \lambda_{j}^{(l)} \geq \sum_{k=1}^{l-1}\left|\mu^{(k)}\right|+\sum_{j=1}^{i} \mu_{j}^{(l)},
$$
for any $1 \leq l \leq \ell$ and any $i \geq 1$. If $\bm{\lambda} \unrhd \bm{\mu}$ and $\bm{\lambda} \neq \bm{\mu}$, then we write $\bm{\lambda} \rhd \bm{\mu}$.

Let $\bm{\lambda}\in\Parts[\ell,n]$. A $\bm{\lambda}$-tableau is a bijective map $\t: [\bm{\lambda}] \mapsto \{ 1, 2,...,n\}$, for example,

\begin{equation*}
\t = \Bigg( \begin{ytableau}
            1 & 2 & 3\\
            4 & 5\\
            6
            \end{ytableau},~~
            \begin{ytableau}
            7\\
            8
            \end{ytableau},~~
            \begin{ytableau}
            9 & 10 & 11\\
            12 & 13
            \end{ytableau} \Bigg)
\end{equation*}
is a $\bm{\lambda}$-tableau, where $\bm{\lambda}=\left((3,2,1),(1,1),(3,2) \right)$ as above. If $\t$ is a $\bm{\lambda}$-tableau, then we set $\Shape(\t):=\bm{\lambda}$, and we call $\mathfrak{t}'=\left(\mathfrak{t}^{(\ell)'}, \cdots, \mathfrak{t}^{(1)'}\right)$ the conjugate of $\t$.

A $\bm{\lambda}$-tableau is standard if its entries increase along each row and each column in each component.
Let $\Std(\bm{\lambda})$ be the set of standard $\blam$-tableaux and $\Std^2(\bm{\lambda}):=\{(\s, \t) \mid \s, \t \in \Std(\bm{\lambda})\}$.

Let $\blam\in\Parts[\ell,n], \t\in\Std(\blam)$ and $1\leq m\leq n$. We use $\t_{\downarrow m}$ to denote the subtableau of $\t$ that contains the numbers $\{1, 2,...,m\}$. If $\t$ is a standard $\bm{\lambda}$-tableau then $\Shape(\t_{\downarrow m})$ is a multipartition for all $m \geq 0$. We define $\s \unrhd \t$ if
and only if $$
\Shape( \s_{\downarrow m} ) \unrhd \Shape (\t_{\downarrow m}),\quad\forall\,1\leq m\leq n.
$$
If $\s \unrhd \t$ and $\s \neq \t$, then write $\s \rhd \t$.

Let $\t^{\bm{\lambda}}$ be the standard $\blam$-tableau which has the numbers $1, 2,\cdots,n$ entered in order from left to right along the rows of $\t^{\lambda^{(1)}}$ and then $\t^{\lambda^{(2)}}, \cdots, \t^{\lambda^{(\ell)}}$. Similarly, let $\t_{\bm{\lambda}}$ be the standard $\blam$-tableau which has the numbers $1, 2, \cdots, n$ entered in order down the columns of $\t^{\lambda^{(\ell)}}, \cdots, \t^{\lambda^{(1)}}$. Given a standard $\bm{\lambda}$-tableau $\t$, we define $d(t), d'(\t)\in \Sym_n$ such that $\t = \t^{\bm{\lambda}} d(t) = \t_{\bm{\lambda}} d'(t)$, and set $w_{\bm{\lambda}}:=d(\t_{\bm{\lambda}})$. For any
$\t\in\Std(\bm{\lambda})$, we have $\t^{\bm{\lambda}} \unrhd \t \unrhd \t_{\bm{\lambda}}$.

\begin{dfn} Let $\bm{\lambda},\bm{\mu}\in\Parts[\ell,n]$, $(\s,\t)\in\Std^2(\bm{\lambda})$ and $(\u,\v)\in\Std^2(\bm{\mu})$.  We define
$$
\text{$(\s,\t)\unrhd (\u,\v)$ if either $\bm{\lam}\rhd\bm{\mu}$ or $\bm{\lam}=\bm{\mu}$, $\s\unrhd\u$ and $\t\unrhd\v$}.
$$
If $(\s,\t)\unrhd (\u,\v)$ and $(\s,\t)\neq (\u,\v)$, the we write $(\s,\t)\rhd(\u,\v)$.
\end{dfn}

The cyclotomic Hecke algebra is a cellular algebra with several different cellular bases. In order recall their constructions we need the following definition.

\begin{dfn}[\text{\rm cf. \cite{DJM2}, \cite{Ma}, \cite{BK08}}]\label{mnlam}
Let $\bm{\mu}\in \P_{\ell,n} $. We define
$$
\begin{aligned}
&\fm_{\bm{\mu}}:=\left(\sum_{w \in \Sym_{\mu}} T_{w}\right)\left(\prod_{k=2}^{\ell} \prod_{m=1}^{\left|\mu^{(1)}\right|+\cdots+\left|\mu^{(k-1)}\right|}\left(L_{m}-Q_{k}\right)\right), \\
&\fn_{\bm{\mu}}:=(-\xi)^{\mathbf{n}(\bmu)}\left(\sum_{w \in \Sym_{\mu'}} (-\xi)^{-\ell(w)} T_{w}\right)\left(\prod_{k=2}^{\ell} \prod_{m=1}^{\left|\mu^{(\ell)}\right|+\left|\mu^{(\ell-1)}\right|+\cdots+\left|\mu^{(\ell-k+2)}\right|}\left(L_{m}-Q_{\ell-k+1}\right)\right),
\end{aligned}
$$
where $\mathbf{n}(\bmu):=\sum_{i=1}^{\ell}(i-1)|\mu^{(i)}|$.
\end{dfn}

\begin{rem} 1) Note that the element $L_k$ in the notation of \cite{MM} is $\mL_k$ in our notations. Our definition of $\fm_\blam$ coincides with that in \cite{Zh} in the case when $\xi=1$, but differs with that in \cite{Ma} by a scalar (because we used different versions of Murphy operators and the cyclotomic parameter $Q_k$ in \cite{Ma} should be identified with $Q_k(\xi-1)+1$ in our notations.

2) our definition of $\fn_{\bmu}$ coincides with the notation $n_{\bmu'}$ in the notation of \cite{Zh} in the case when $\xi=1$ and differs with
the notation $n_{\bmu'}$ in \cite{Ma} by a scalar. Our slightly looking different definition of $\fn_\bmu$ can be better understood from Definition \ref{nst1}, Lemma \ref{cellular1} and Remark \ref{ind1} below.
\end{rem}

\begin{dfn}[{\cite{DJM2}}]\label{nst1}
Fix $\bm{\lambda}\in\P_{\ell,n}$, for any $\s,\t\in \Std(\bm{\lambda})$ define
$$
\begin{aligned}
\fm_{\s \t}:&=\left(T_{d(\mathfrak{\s})}\right)^{*} \mathfrak{m}_{\bm{\lambda}} T_{d(\t)},\\
\fn_{\s \t}:&=(-\xi)^{-\ell\left(d'(\s)\right)-\ell\left(d'(\t)\right)} (T_{d (\s')})^* \mathfrak{n}_{\bm{\lambda}} T_{d(\t')}.
\end{aligned}
$$
\end{dfn}

\begin{lem}[\cite{DJM2,MM}]\label{cellular1} The set $$\{\fm_{\s \t} \mid \s, \t \in \Std(\bm{\lambda}), \bm{\lambda} \in \P_{\ell, n}\}$$ form a cellular basis
of $\HH_{n}$ with respect to the poset $(\P_{\ell, n},\unrhd)$; while the set $$\{\fn_{\s \t} \mid \s, \t \in \Std(\bm{\lambda}), \bm{\lambda} \in \P_{\ell, n}\}$$ form a cellular basis
of $\HH_{n}$ with respect to the poset $(\P_{\ell, n},\unlhd)$.
\end{lem}

Sometimes in order to emphasize the ground ring $R$ we shall use the notation $\fm_{\s\t}^R, \fn_{\s\t}^R$ instead of $\fm_{\s\t},\fn_{\s\t}$.

\begin{rem}\label{ind1} Let $q,\dot{Q}_1,\cdots,\dot{Q}_\ell$ be indeterminates over $\Z$. We set $$
\dot{q}:=\begin{cases} q, &\text{if $\xi\neq 1$;}\\ 1, &\text{if $\xi=1$.}
\end{cases}
$$
Set $\mathscr{A}:=\Z[\dot{q}^{\pm 1}, \dot{Q}_1,\cdots,\dot{Q}_\ell]$. Let $\mathcal{F}:=\Q(\dot{q}, \dot{Q}_1,\cdots,\dot{Q}_\ell)$.
Let $\HH_{n,\mathscr{A}}(\dot{q},\dot{\bQ})$ be the cyclotomic Hecke algebra of type $G(\ell,1,n)$ over $\mathscr{A}$ with Hecke parameter $\dot{q}$ and cyclotomic parameters $\dot{\bQ}:=(\dot{Q}_1,\cdots,\dot{Q}_\ell)$.
Set $\HH_{n,\mathcal{F}}(\dot{q},\dot{\bQ}):=\mathcal{F}\otimes_{\mathscr{A}}\HH_{n,\mathscr{A}}(\dot{q},\dot{\bQ})$. Then $\HH_{n,\mathcal{F}}(\dot{q},\dot{\bQ})$ is split semisimple.
We set $'$ to be the unique ring involution  of $\HH_{n,\mathscr{A}}(\dot{q},\dot{\bQ})$ (\cite[\S3]{Ma}) which is defined on generators by $$
L'_m=-\dot{q} L_m,\,\, {T}'_i:=-\dot{q}^{-1}{T}_i,\,\, \dot{q}':=\dot{q}^{-1},\,\,\dot{Q}'_j:=-\dot{q}\dot{Q}_{\ell-j+1},\quad \forall\,1\leq i<n,\, 1\leq j\leq\ell, 1\leq m\leq n .
$$
Clearly, $'$ naturally extends to a ring involution  of $\HH_{n,\mathscr{A}}(\dot{q},\dot{\bQ})$. We have that $\fm'_{\s\t}=\fn_{\s'\t'}$,
\end{rem}

Let $K$ be a field and $x$ be an indeterminate element over $K$. Set $\O_K:= K[x]_{(x)}$, $\K := K(x)$. We define \begin{equation}\label{OHecke}
\hat{Q}_1:=x^n+Q_1,\,\,\,\hat{Q}_2:=x^{2n}+Q_2, \cdots , \hat{Q}_\ell:=x^{\ell n}+Q_\ell,\,\,\, \hat{\xi}:=x+\xi.
\end{equation}
For $R\in\{\O_K,\K\}$, we denote the cellular bases of $\HH_{n,R}=\HH_{n,R}\left(\hat{\xi}, \hat{Q}_{1}, \cdots, \hat{Q}_{\ell}\right)$
by $\{\fm^{R}_{\s \t} \mid \s, \t \in \Std(\bm{\lambda}), \bm{\lambda} \in \P_{\ell, n}\}$. Since $\HH_{n,\O_K}$ is a free $\O_K$-module, the natural map
$\HH_{n,\O_K}\rightarrow \K\otimes_{\O_K}\HH_{n,\O_K}\cong\HH_{n,\K}$ is injective. Henceforth, we shall identify $\HH_{n,\O_K}$ with its image under this injection. In particular, $\fm_{\s\t}^{\O_K}$ will be identified with $\fm_{\s\t}^{\K}$. By Lemma \ref{ss}, we see that $\HH_{n,\K}$ is split semisimple.

For any $\t=(\t^{(1)},\cdots,\t^{(\ell)}) \in \Std(\bm{\lambda})$ and $1\leq k\leq n$, if $k$ appears in $i$-th row $j$-th column of $\t^{(c)}$ then we define
\begin{equation}\label{content}
c_{\t}(k)=\cont(\gamma):=\hat{\xi}^{j-i}\hat{Q}_c+(1+\hat{\xi}+\cdots+\hat{\xi}^{j-i-1}) ,\quad\text{where $\gamma:=(c,i,j)$.}
\end{equation}
We also define $C(k):=\left\{c_{\t}(k) \mid \t \in \Std(\bm{\lambda}), \bm{\lambda} \in \P_{\ell, n}\right\}$.

\begin{dfn}[{\cite{Ma}}]\label{dfn:Ft}
Suppose $\bm{\lam} \in \P_{\ell, n}$ and  $\t\in\Std(\bm{\lam})$. We define
$$
F_{\t}=\prod\limits^n\limits_{k=1}\prod\limits_{\substack{c\in C(k)\\c\neq c_{\t}(k)}}\frac{L_k-c}{c_{\t}(k)-c}.
$$
For any $\bm{\lam} \in \P_{\ell, n}$ and $\s,\t\in\Std(\bm{\lam})$, we define
$$
f_{\s\t}:=F_\s \fm_{\s\t}F_\t,\quad g_{\s\t}:=F_{\s}\fn_{\s\t}F_{\t}.
$$
\end{dfn}\noindent

\begin{rem}\label{convention} Note that our convention for the notations $\fn_{\s\t}, g_{\s\t}$ in this paper agrees with the one used in \cite{HW}, which differs with the corresponding notations in \cite{Ma} by a conjugation and an invertible scalar. The elements $\fn_{\s\t}, \mfg_{\s\t}$ in the current paper should be compared with the elements $n_{\s'\t'}, g_{\s'\t'}$ in \cite{Ma} up to some invertible scalar. In particular, our dual cellular basis $\{\fn_{\s\t}\}$ use the partial order ``$\unlhd$'', while \cite{Ma} use the partial order ``$\unrhd$'' for the dual cellular basis.
\end{rem}

\begin{lem}\text{(\cite[2.6, 2.11]{Ma},\cite[Proposition 6.8]{AMR})}\label{fsemi}  1) For any $\s,\t\in\Std(\blam), \u,\v\in\Std(\bmu)$, we have $$
f_{\s\t}f_{\u\v}=\delta_{\t\u}\gamma_\t f_{\s\v},
$$
for some $\gamma_\t\in\K^\times$;

2) The set $\bigl\{f_{\s\t}\bigm|\s,\t\in\Std(\blam),\blam\in\Parts[\ell,n]\bigr\}$ is a $\K$-basis of $\HH_{n,\K}$.

3) We have $$\begin{aligned}
\fm_{\s\t}&\equiv f_{\s\t}+\sum_{\substack{\u,\v\in\Std(\blam)\\ (\u,\v)\rhd(\s,\t)}}a_{\u\v}^{\s\t}f_{\u\v}\,\,\pmod{(\HH_{n,\K})^{\rhd\blam}},\\
f_{\s\t}&\equiv \fm_{\s\t}+\sum_{\substack{\u,\v\in\Std(\blam)\\ (\u,\v)\rhd(\s,\t)}}b_{\u\v}^{\s\t}\fm_{\u\v}\,\,\pmod{(\HH_{n,\K})^{\rhd\blam}},
\end{aligned}
$$
where $a_{\u\v}^{\s\t}, b_{\u\v}^{\s\t}\in\K$ for each pair $(\u,\v)$.
\end{lem}

The basis $\bigl\{f_{\s\t}\bigm|\s,\t\in\Std(\blam),\blam\in\Parts[\ell,n]\bigr\}$ is called {\it the seminormal basis} of $\HH_{n,\K}$.

\begin{lem}\text{(\cite[\S3]{Ma}, \cite[\S3,\S4]{HW})}\label{gsemi}  1) For any $\s,\t\in\Std(\blam), \u,\v\in\Std(\bmu)$, we have $$
g_{\s\t}g_{\u\v}=\delta_{\t\u}\check{\gamma}_{\t'} g_{\s\v},
$$
for some $\check{\gamma}_{\t'}\in\K^\times$;

2) The set $\bigl\{g_{\s\t}\bigm|\s,\t\in\Std(\blam),\blam\in\Parts[\ell,n]\bigr\}$ is a $\K$-basis of $\HH_{n,\K}$.

3) We have $$\begin{aligned}
\fn_{\s\t}&\equiv g_{\s\t}+\sum_{\substack{\u,\v\in\Std(\blam)\\ (\u,\v)\lhd(\s,\t)}}\check{a}_{\u\v}^{\s\t}g_{\u\v}\,\,\pmod{(\HH_{n,\K})^{\lhd\blam}},\\
g_{\s\t}&\equiv \fn_{\s\t}+\sum_{\substack{\u,\v\in\Std(\blam)\\ (\u,\v)\lhd(\s,\t)}}\check{b}_{\u\v}^{\s\t}\fn_{\u\v}\,\,\pmod{(\HH_{n,\K})^{\lhd\blam}},
\end{aligned}
$$
where $\check{a}_{\u\v}^{\s\t}, \check{b}_{\u\v}^{\s\t}\in\K$ for each pair $(\u,\v)$.
\end{lem}

The basis $\bigl\{g_{\s\t}\bigm|\s,\t\in\Std(\blam),\blam\in\Parts[\ell,n]\bigr\}$ is called {\it the dual seminormal basis} of $\HH_{n,\K}$.

Let $\mC$ be a subset of the set of all multicompositions of $n$ which have $\ell$ components such that if $\bmu\in\mC$ and $\bnu\in\Parts[\ell,n]$ with $\bnu\rhd\bmu$ then $\bnu\in\mC$. We let $\mC^+$ be the set of multipartitions in $\mC$. Next we want to recall some basic knowledge about cyclotomic Schur algebras and their seminormal bases.

\begin{dfn}[\cite{DJM2}, \cite{BK08}]\label{CSchur} The cyclotomic Schur algebra $S_{n,R}$ over $R$ is defined to be the following endomorphism algebra: $$
S_{n,R}:=\End_{\HH_{n,R}}\Bigl(\bigoplus_{\bmu\in\mC}\fm_\bmu^R\HH_{n,R}\Bigr).
$$
If $\O\in\{K,\O_K,\K\}$, we use $S_{n,\O}$ to denote the cyclotomic Schur algebra corresponding to the cyclotomic Hecke algebra $\HH_{n,\O}$ with Hecke parameters and cyclotomic parameters given in (\ref{OHecke}).
\end{dfn}

Let $\blam\in\Parts[\ell,n]$ and $\bmu\in\mC$. A $\blam$-tableau $\bS$ has type $\bmu$ if its entries are ordered pairs $(i,k)$, where $1\leq i\leq n, 1\leq k\leq\ell$, such that for all $i$ and $k$ the number of times $(i,k)$ is an entry in $\bS$ is $\mu_i^{(k)}$. Given two pair $(i_1,k_1), (i_2,k_2)$, we write
$(i_1,k_1)<(i_2,k_2)$ if $k_1<k_2$, or $k_1=k_2$ and $i_1<i_2$. A $\blam$-tableau $\bS=(\bS^{(1)},\cdots,\bS^{(\ell)})$ of type $\bmu$ is said to be semistandard if \begin{enumerate}
\item the entries in each row of each component $\bS^{(k)}$ of $\bS$ are non–decreasing;
and
\item the entries in each column of each component $\bS^{(k)}$ of $\bS$ are strictly increasing;
and
\item for each $1\leq k\leq\ell$, no entry in $\bS^{(k)}$ has the form $(i,l)$ with $l < k$.
\end{enumerate}
We use $\mT(\blam,\bmu)$ to denote the set of semistandard $\blam$-tableaux of type $\bmu$. By definition, $\mT(\blam,\bmu)\neq\emptyset$ only if $\blam\unrhd\bmu$. Set $\omega:=(\emptyset,\cdots,\emptyset,(1^n))$. Then we have $\mT(\blam,\omega)=\Std(\blam)$. For any $\s\in\Std(\blam)$, we define $\bmu(\s)$ to be the $\blam$-tableau obtained from $\s$ by replacing each entry $m$ in $\s$ by $(i,k)$ if $m$ is in row $i$ of the $k$th component of $\t^\bmu$. Then $\bmu(\s)\in\mT(\blam,\bmu)$.

Let $\bmu,\bnu\in\mC$ and $\blam\in\Parts[\ell,n]$. For any $\bS\in\mT(\blam,\bmu), \bT\in\mT(\blam,\bnu)$, we define \begin{equation}
\fm_{\bS\bT}^R:=\sum_{\substack{\s,\t\in\Std(\blam)\\ \bmu(\s)=\bS, \bnu(\t)=\bT}}\fm_{\s\t}^R .
\end{equation}

\begin{lem}\text{(\cite[Proposition 6.3]{DJM2})}\label{fmST} Let $\bmu,\bnu\in\mC$ and $\blam\in\Parts[\ell,n]$. Then for any $\bS\in\mT(\blam,\bmu), \bT\in\mT(\blam,\bnu)$,
$\fm_{\bS\bT}^R\in\HH_{n,R}\fm_{\bnu}^R\cap \fm_{\bmu}^R\HH_{n,R}$.
\end{lem}

\begin{dfn}\text{(\cite[Theorem 6.6]{DJM2})}  Let $\bmu,\bnu\in\mC$, $\blam\in\Parts[\ell,n]$, $\bS\in\mT(\blam,\bmu)$ and $\bT\in\mT(\blam,\bnu)$. We define $$
\varphi^R_{\bS\bT}\in\Hom_{\HH_{n,R}}\bigl(\fm_{\bnu}^R\HH_{n,R}, \fm_{\bmu}^R\HH_{n,R}\bigr)$$ by $$
\varphi^R_{\bS\bT}(\fm_{\bnu}^R h):=\fm_{\bS\bT}^R h,\quad\forall\,h\in \HH_{n,R} .
$$
Extend $\varphi^R_{\bS\bT}$ to an element of $S_{n,R}$ by defining $\varphi^R_{\bS\bT}$ to be zero on $\fm_{\brho}^R\HH_{n,R}$ whenever $\bnu\neq\brho\in\mC$.
\end{dfn}

Recall that $\omega=(\emptyset,\cdots,\emptyset,(1^n))$.

\begin{lem}\label{varphim1}  Assume that $\omega\in\mC^+$. Let $\blam\in\Parts[\ell,n]$. Then there is a natural isomorphism between the centralizer subalgebra $\phi^1_{\omega\omega}S_{n,R}  \phi^1_{\omega\omega}$ of $S_{n,R} $ and the cyclotomic Hecke algebra $\HH_{n,R} $, where $\phi^1_{\omega\omega}$ is defined in Definition~\ref{dfn:cenopr} below. Furthermore, for any $\s,\t\in\Std(\blam)=\mathcal{T}_0(\blam,\omega)$, $\varphi_{\s\t}^R$ is the same as the endomorphism of $\HH_{n,R}^\lam$ given by left multiplication with the element $\fm_{\s\t}^R$.
\end{lem}

\begin{proof} This follows directly from the definition of $\varphi^R_{\bS\bT}$.
\end{proof}

\begin{lem}[\text{\cite[Theorem 6.6]{DJM2}}] The following set \begin{equation}\label{stdbasis}
\bigl\{\varphi^R_{\bS\bT}\bigm|\bS\in\mT(\blam,\bmu), \bT\in\mT(\blam,\bnu),\bmu,\bnu\in\mC,\blam\in\mC^+\bigr\}
\end{equation}
is a cellular $R$-basis of $S_{n,R}$. If $R=K$, then $S_{n,K}$ is a quasi-hereditary algebra over $K$.
\end{lem}
The cellular basis (\ref{stdbasis}) is called the {\bf semistandard basis} of the cyclotomic Schur algebra $S_{n,R}$.

For each $\blam\in\mC^+$, we define $S_{n,R}^{\rhd\blam}$ to be the $R$-submodule of $S_{n,R}$ spanned by $$
\bigl\{\varphi^R_{\bS\bT}\bigm|\bS\in\mT(\brho,\bmu), \bT\in\mT(\brho,\bnu),\bmu,\bnu\in\mC,\blam\lhd\brho\in\mC^+\bigr\}.
$$
By \cite[Theorem 6.6]{DJM2}, $S_{n,R}^{\rhd\blam}$ is a two-sided ideal of $S_{n,R}$.

\begin{dfn}\text{(\cite[Definition 6.13]{DJM2})}  Let $\blam\in\mC^+$. The Weyl module $\Delta_R^\blam$ is the left $S_{n,R}$-submodule of $S_{n,R}/S_{n,R}^{\rhd\blam}$ given by $$
\Delta_R^\blam:=S_{n,R}\Bigl(\varphi_{\bT^\blam\bT^\blam}+(S_{n,R})^{\rhd\blam}\Bigr).
$$
\end{dfn}
For any $\bS\in\mT(\blam,\bmu)$, we set $\varphi^R_\bS:=\varphi^R_{\bS\bT^\blam}+(S_{n,R})^{\rhd\blam}$. Then $\{\varphi^R_\bS|\bS\in\mT(\blam,\bmu),\bmu\in\mC\}$ is an $R$-basis of the Weyl module $\Delta_R^\blam$.
There is a natural bilinear form on $\Delta_R^\blam$ which satisfies that for any $\bS,\bT\in\mT(\blam,\bmu), \bU\in\mT(\blam,\balpha)$ and $\bV\in\mT(\blam,\bnu)$ , \begin{equation}\label{bilinear1}
\varphi^R_{\bU\bS}\varphi^R_{\bT\bV}\equiv \<\varphi^R_\bS,\varphi^R_\bT\>\varphi^R_{\bU\bV}\,\pmod{S_{n,R}^{\rhd\blam}} .
\end{equation}
In particular, by the proof of \cite[Theorem 6.16]{DJM2} (and a similar calculation in the degenerate case), we have \begin{equation}\label{bilinear2}
\<\varphi^R_{\bT^\blam}, \varphi^R_{\bT^\blam}\>=1 .
\end{equation}

\bigskip
\section{The unified trace form on cyclotomic Hecke algebra}
\label{sec.unified.traceform}

In this section we shall give a unified definition of symmetrizing form $\tau$ on $\HH_{n}$ and also give a self-contained proof that $\HH_{n}$ is a symmetric algebra.

For later use, we need the following lemma, which is an analogue of the result in \cite[Lemma 3.3]{MM} for the Jucys-Murphy operators $\{L_k\}$ (while \cite[Lemma 3.3]{MM} used the Jucys-Murphy operators $\{\mL_k\}$).

\begin{lem}\label{commurel2} Let $1\leq i<n$ be an integer. Let $k,m\in\N$. \begin{enumerate}
\item If $k \geq m$ then $$
    T_{i} L_{i+1}^{m} L_{i}^{k} =  L_{i+1}^{k} L_{i}^{m} T_{i} - (\xi-1) \sum_{c=m+1}^{k} L_{i+1}^{c} L_{i}^{k+m-c} - \sum_{c=m}^{k-1} L_{i+1}^{c} L_{i}^{k+m-c-1}.
$$
\item If $k \leq m$ then $$
    T_{i} L_{i+1}^{m} L_{i}^{k} = L_{i+1}^{k} L_{i}^{m} T_{i} + (\xi-1)\sum_{c=k+1}^{m} L_{i+1}^{c} L_{i}^{k+m-c} + \sum_{c=k}^{m-1} L_{i+1}^{c} L_{i}^{k+m-c-1}.
$$
\end{enumerate}
\end{lem}

\begin{proof} a) Suppose that $k\geq m$. We use induction on $k-m$. If $k-m=0$, then as $T_i$ commutes with $L_iL_{i+1}$ we have that $T_{i} L_{i+1}^{m} L_{i}^{k} =  L_{i+1}^{k} L_{i}^{m} T_{i}$.
If $k-m =1$, then
$$  \begin{aligned}
    T_{i} L_{i+1}^{m} L_{i}^{m+1} & = T_{i} (L_{i+1}^{m} L_{i}^{m}) L_i = (L_{i+1}^{m} L_{i}^{m}) (T_{i} L_{i})  = L_{i+1}^{m} L_{i}^{m} (L_{i+1}T_{i} - (\xi-1)L_{i+1} -1) \\
    & = L_{i+1}^{m+1} L_{i}^{m} T_{i} -(\xi-1) L_{i+1}^{m+1}L_{i}^{m} - L_{i+1}^{m}L_{i}^{m} .
  \end{aligned} $$
Suppose the statement in part a) of the lemma holds when $k-m= t\in\Z^{\geq 1}$. That says,
$$
T_{i} L_{i+1}^{m} L_{i}^{m+t} =  L_{i+1}^{m+t} L_{i}^{m} T_{i} - (\xi-1) \sum_{c=m+1}^{m+t} L_{i+1}^{c} L_{i}^{m+t+m-c} - \sum_{c=m}^{m+t-1} L_{i+1}^{c} L_{i}^{m+t+m-c-1}.
$$
Now assume that $k-m= t+1$. Then we have $$
\begin{aligned}
   &\quad\, T_{i} L_{i+1}^{m} L_{i}^{m+t+1}\\
& =  (T_{i}L_{i+1}^{m} L_{i}^{m+t}) L_{i}\\
& =\Bigl(L_{i+1}^{m+t} L_{i}^{m} T_{i} - (\xi-1) \sum_{c=m+1}^{m+t} L_{i+1}^{c} L_{i}^{m+t+m-c} - \sum_{c=m}^{m+t-1} L_{i+1}^{c} L_{i}^{m+t+m-c-1}\Bigr)L_{i}\,\,\,\text{(by induction hypothesis)}  \\ \notag
& =  L_{i+1}^{m+t} L_{i}^{m} T_{i}L_i - (\xi-1) \sum_{c=m+1}^{m+t} L_{i+1}^{c} L_{i}^{m+t+m-c+1} - \sum_{c=m}^{m+t-1} L_{i+1}^{c} L_{i}^{m+t+m-c}\\
&=L_{i+1}^{m+t} L_{i}^{m}\bigl(L_{i+1}T_i-(\xi-1)L_{i+1}-1\bigr) - (\xi-1) \sum_{c=m+1}^{m+t} L_{i+1}^{c} L_{i}^{m+t+m-c+1} - \sum_{c=m}^{m+t-1} L_{i+1}^{c} L_{i}^{m+t+m-c}\\
&=L_{i+1}^{m+t+1} L_{i}^{m} T_{i} - (\xi-1) \sum_{c=m+1}^{m+t+1} L_{i+1}^{c} L_{i}^{m+t+1+m-c} - \sum_{c=m}^{m+t} L_{i+1}^{c} L_{i}^{m+t+1+m-c-1}.
  \end{aligned} $$
So we are done. This proves part a) of the lemma.

b) Suppose $m\geq k$. By a), we have that $$
T_{i} L_{i+1}^{k} L_{i}^{m} =  L_{i+1}^{m} L_{i}^{k} T_{i} - (\xi-1) \sum_{c=k+1}^{m} L_{i+1}^{c} L_{i}^{k+m-c} - \sum_{c=k}^{m-1} L_{i+1}^{c} L_{i}^{k+m-c-1}.
$$
Applying the anti-involution $\ast$ to a), we get that $$
L_{i+1}^{k} L_{i}^{m}T_{i}  = T_{i} L_{i+1}^{m} L_{i}^{k}  - (\xi-1) \sum_{c=k+1}^{m} L_{i+1}^{c} L_{i}^{k+m-c} - \sum_{c=k}^{m-1} L_{i+1}^{c} L_{i}^{k+m-c-1},
$$
which is equivalent to $$
T_{i} L_{i+1}^{m} L_{i}^{k} = L_{i+1}^{k} L_{i}^{m} T_{i} + (\xi-1)\sum_{c=k+1}^{m} L_{i+1}^{c} L_{i}^{k+m-c} + \sum_{c=k}^{m-1} L_{i+1}^{c} L_{i}^{k+m-c-1}.
$$
This proves part b) of the lemma.
\end{proof}

\begin{cor}
  Suppose that $1<i \leq n$ and $k \geq 1$. Then
    \begin{equation} \notag
        L_{i}^{k}=\xi^{-1} T_{i-1} L_{i-1}^{k} T_{i-1}+\xi^{-1}(\xi-1) \sum_{c=1}^{k-1} L_{i}^{c} L_{i-1}^{k-c} T_{i-1} + \xi^{-1} \sum_{c=0}^{k-1} L_{i-1}^{c} L_{i}^{k-1-c} T_{i-1}.
    \end{equation}
\end{cor}

\begin{proof}
We prove the statement by induction on $k$.
The case when $k=1$ is clear as the it follows directly from the defining relation of $\HH_{n,R}$. In general, assume $k\geq 1$. Taking $m=0$ in part a) of Lemma \ref{commurel2}, we see that
  \begin{equation} \notag
    \begin{aligned}
         L_{i}^{k+1} &=\xi^{-1} L_{i}^{k} T_{i-1} L_{i-1} T_{i-1} +\xi^{-1} L_{i}^{k} T_{i-1} \\
          &=\xi^{-1}\left\{T_{i-1} L_{i-1}^{k}+(\xi-1) \sum_{c=1}^{k} L_{i}^{c} L_{i-1}^{k-c} + \sum_{c=0}^{k-1} L_{i-1}^{c} L_{i}^{k-c-1} \right\} L_{i-1} T_{i-1} +\xi^{-1} L_{i}^{k} T_{i-1} \\
          &=\xi^{-1} T_{i-1} L_{i-1}^{k+1} T_{i-1}+\xi^{-1}(\xi-1) \sum_{c=1}^{k} L_{i}^{c} L_{i-1}^{k+1-c} T_{i-1} + \xi^{-1} \sum_{c=0}^{k} L_{i-1}^{c} L_{i}^{k-c} T_{i-1}.
    \end{aligned}
  \end{equation}
This proves the corollary.
\end{proof}

Our purpose is to give a unified definition of the symmetrizing form $\tau$ on $\HH_{n,R}$ which generalizes both $\tau^{\MM}$ and $\tau^{\BK}$. Before doing this, we first give some examples to show that some ``native'' generalization of $\tau^{\MM}$ and $\tau^{\BK}$ dos not work.

One ``native'' option is to mimic the definition of $\tau^{\MM}$ to define a $K$-linear function on $\HH_{n}$ as follows:
$$
\tau^{\MM'}(L_1^{c_1}\cdots L_n^{c_n}T_w):=
\begin{cases}
1, &\text{if $w=1$ and $c_1=\cdots=c_n=0$, }\\
0, &\text{otherwise.}
\end{cases}
$$
We extend the map $\tau^{\MM'}$ linearly to an $R$-linear function on $\HH_{n,R}$. The following example shows that this linear function is not a symmetric form.

\begin{exmp}
Suppose $\xi = 1$. Then in the degenerate cyclotomic Hecke algebra $H_{2,2}$ we have $$
\tau^{\MM'}(L_1T_1)  =0, \quad \,\tau^{\MM'}(T_1L_1)  =\tau^{\MM'}(L_2T_1-1)=-1.
$$
\end{exmp}

Another ``native'' option is simply to use $\tau^{\BK}$ to define a linear function on $\HH_{n}$. In the following example we show that this linear function is not a symmetric form.

\begin{exmp}
Suppose $\xi \neq 1$. In the non-degenerate cyclotomic Hecke algebra $\mathcal{H}_{2,2}$, we have
$$\begin{aligned}
\tau^{\BK}(L_1L_1T_1) & =\tau^{\BK}(Q_1L_1T_1+Q_2L_1T_1-Q_1Q_2T_1)=0, \\
\tau^{\BK}(L_1T_1L_1) & =\tau^{\BK}(L_1L_2T_1+(1-\xi)L_1L_2-L_1)=1-\xi .
\end{aligned} $$
\end{exmp}

In the following we shall give our unified definition of the symmetrizing form $\tau$ on $\HH_{n,R}$.

\begin{dfn}\label{trace} For any $w\in\Sym_n$ and integers $0\leq c_1,c_2,\cdots,c_n<\ell$, we define $$
\tau(L_1^{c_1}\cdots L_n^{c_n}T_w):=\begin{cases} (1-\xi)^{n(\ell-1)-\sum_{i=1}^n c_i}, &\text{if $w=1$;}\\
    0, &\text{otherwise,}
\end{cases}
$$
where we use the convention that $0^0 :=1$ when $\xi = 1$.
We extends the map $\tau$ linearly to a $K$-linear function on $\HH_{n,R}$.
\end{dfn}
By construction, we see that $\tau$ coincides with $\tau^{\BK}$ when $\xi=1$. For any $x,y\in\HH_{n,R}$, we write $x\overset{\tau}{=}y$ whenever $\tau(x)=\tau(y)$. It is clear that $\tau$ induces a bilinear form on $\HH_{n,R}$ via: $(x,y)\mapsto\tau(xy)$, $\forall\,x,y\in\HH_{n,R}$.

Note that in the non-degenerate case, the proof that $\tau^{\MM}$ is a symmetrizing form is contained in two papers \cite{BM} and \cite{MM}. The paper \cite{BM} proved the linear form $\tau^{\MM}$ is symmetric and non-degenerate for the semisimple cyclotomic Hecke algebras, while the paper \cite{MM} proved the form $\tau^{\MM}$ is non-degenerate for the cyclotomic Hecke algebras over arbitrary ground rings and with arbitrary parameters. In this paper, we shall give a new proof that our unified definition \ref{trace} does give a symmetrizing form on $\HH_{n,R}$.

\begin{thm}\label{mainthm01}
For any $x,y\in\HH_{n,R}$, we have $\tau(xy)=\tau(yx)$. In other words, the bilinear form on $\HH_{n,R}$ induced from $\tau$ is a symmetric.
\end{thm}

\begin{proof} We divide the proof into three steps:

\smallskip
\noindent
{\it Step 1.} We claim that for any $x\in\HH_{n,R}$ and $1\leq r<n$, $\tau(xT_r)=\tau(T_r x)$. To prove this claim, it suffices (by Lemma \ref{stdBasis}) to show that for any $w\in\Sym_n$, $1\leq r<n$ and integers $0\leq a_1,\cdots,a_n<\ell$, \begin{equation}\label{commu1}
L_{1}^{a_1} \cdots L_{n}^{a_n} T_w T_r\equal T_r L_{1}^{a_1} \cdots L_{n}^{a_n} T_w .
\end{equation}


Suppose $w = 1$. Then $T_{r} L_{1}^{a_1} \cdots L_{n}^{a_n}=L_{1}^{a_1} \cdots L_{r-1}^{a_{r-1}} (T_r L_{r}^{a_r} L_{r+1}^{a_{r+1}})L_{r+2}^{a_{r+2}} \cdots L_{n}^{a_{n}}$.
If $a_r\geq a_{r+1}$ then by Lemma \ref{commurel2} we have $$\begin{aligned}
T_{r} L_{1}^{a_1} \cdots L_{n}^{a_n} &= ( L_{1}^{a_1} \cdots L_{r-1}^{a_{r-1}} ) \Bigl(L_{r}^{a_{r+1}} L_{r+1}^{a_r} T_{r} - (\xi-1) \sum_{c=a_{r+1}+1}^{a_r} L_{r+1}^{c} L_{r}^{a_r+a_{r+1}-c} \\
        & \qquad\qquad - \sum_{c=a_{r+1}}^{a_r-1} L_{r+1}^{c} L_{r}^{a_r+a_{r+1}-c-1} \Bigr) L_{r+2}^{a_{r+2}} \cdots L_{n}^{a_{n}}.
\end{aligned}
$$
Note that in the first summation above, we have $c\leq a_r<\ell, a_r+a_{r+1}-c\leq a_r-1<\ell$, thus we have $$\begin{aligned}
&\quad\,-(\xi-1)\tau\Bigl(L_{1}^{a_1} \cdots L_{r-1}^{a_{r-1}}L_{r+1}^{c} L_{r}^{a_r+a_{r+1}-c}L_{r+2}^{a_{r+2}} \cdots L_{n}^{a_{n}}\Bigr)\\
&=-(\xi-1)(1-\xi)^{n(\ell-1)-\sum_{j=1}^n a_j}=(1-\xi)^{n(\ell-1)-\sum_{j=1}^n a_j+1};
\end{aligned}$$
while in the second summation above, $c\leq a_r-1<\ell-1,\quad a_r+a_{r+1}-c-1\leq a_r-1<\ell$, thus we have
$$
-\tau\Bigl(L_{1}^{a_1} \cdots L_{r-1}^{a_{r-1}}L_{r+1}^{c} L_{r}^{a_r+a_{r+1}-c-1}L_{r+2}^{a_{r+2}} \cdots L_{n}^{a_{n}}\Bigr)=-(1-\xi)^{n(\ell-1)-\sum_{j=1}^n a_j+1}.
$$
The two terms cancel off each other. It follows that $T_{r} L_{1}^{a_1} \cdots L_{n}^{a_n} \equal  L_{1}^{a_1} \cdots L_{n}^{a_n} T_{r}$ as required.
The case that $a_r \leq a_{r+1}$ can be proved in a similar way.

Suppose $1\neq w\neq s_r$. It is clear that $T_wT_r=\sum_{1\neq x\in\Sym_n}c_xT_x$. It follows  from Definition \ref{trace} that $$
\tau(L_{1}^{a_1} \cdots L_{n}^{a_n}T_w T_{r})=\sum_{1\neq x\in\Sym_n}\tau(L_{1}^{a_1} \cdots L_{n}^{a_n}c_xT_x)=0 .
$$
On the other hand, $T_rL_{1}^{a_1} \cdots L_{n}^{a_n}T_w=L_{1}^{a_1} \cdots L_{r-1}^{a_{r-1}} (T_r L_{r}^{a_r} L_{r+1}^{a_{r+1}})L_{r+2}^{a_{r+2}} \cdots L_{n}^{a_n}T_w$. Applying Lemma \ref{commurel2}, we see that $T_rL_{1}^{a_1} \cdots L_{n}^{a_n}T_w$ is spanned by some elements of the form $$
L_{1}^{c_1} \cdots L_{n}^{c_n}T_w,\quad\, L_{1}^{c_1} \cdots L_{n}^{c_n}T_rT_w,
$$
where $0\leq c_i\leq\ell-1$ for each $i$. All these terms make zero contribution to $\tau$. So we still have $T_{r} L_{1}^{a_1} \cdots L_{n}^{a_n}T_w \equal  L_{1}^{a_1} \cdots L_{n}^{a_n}T_wT_{r}$ in this case.

Suppose $w=s_r$. Then $T_wT_r=T_r^2=(\xi-1)T_r+\xi$.  It follows  from Definition \ref{trace} that $$
\tau(L_{1}^{a_1} \cdots L_{n}^{a_n}T_w T_{r})=\xi\tau(L_{1}^{a_1} \cdots L_{n}^{a_n}) .
$$
On the other hand, $T_rL_{1}^{a_1} \cdots L_{n}^{a_n}T_w=L_{1}^{a_1} \cdots L_{r-1}^{a_{r-1}} (T_r L_{r}^{a_r} L_{r+1}^{a_{r+1}})L_{r+2}^{a_{r+2}} \cdots L_{n}^{a_n}T_r$. Applying Lemma \ref{commurel2} and Definition \ref{trace}, we see that $$\begin{aligned}
\tau(T_rL_{1}^{a_1} \cdots L_{n}^{a_n}T_w)
&=\tau\Bigl(L_{1}^{a_1} \cdots L_{r-1}^{a_{r-1}} (L_{r+1}^{a_r} L_{r}^{a_{r+1}})L_{r+2}^{a_{r+2}} \cdots L_{n}^{a_n}T_r^2\Bigr)\\
&=\tau\Bigl(L_{1}^{a_1} \cdots L_{r-1}^{a_{r-1}} (L_{r+1}^{a_r} L_{r}^{a_{r+1}})L_{r+2}^{a_{r+2}} \cdots L_{n}^{a_n}((\xi-1)T_r+\xi)\Bigr)\\
&=\xi\tau\Bigl(L_{1}^{a_1} \cdots L_{r-1}^{a_{r-1}} (L_{r+1}^{a_r} L_{r}^{a_{r+1}})L_{r+2}^{a_{r+2}} \cdots L_{n}^{a_n}\Bigr)\\
&=\xi\tau(L_{1}^{a_1} \cdots L_{n}^{a_n})=\tau(L_{1}^{a_1} \cdots L_{n}^{a_n}T_w T_{r}),\end{aligned}
$$
as required. Thus we have proved that (\ref{commu1}) for any $1\leq r<n$.

\medskip
\noindent
{\it Step 2.} We claim that for any $x\in\HH_{n,R}$, $\tau(xL_n)=\tau(L_n x)$.

To prove this claim, it suffices to show that for any $w\in\Sym_n$, and integers $0\leq a_1,\cdots,a_n<\ell$, \begin{equation}\label{commu3}
L_{1}^{a_1} \cdots L_{n}^{a_n} T_w L_n\equal L_n L_{1}^{a_1} \cdots L_{n}^{a_n} T_w=L_{1}^{a_1} \cdots L_{n}^{a_n+1} T_w .
\end{equation}

If $w\in\Sym_{n-1}$ then $T_w$ commutes with $L_n$ and there is nothing to prove. Henceforth, we assume that $w\not\in\Sym_{n-1}$. We can write $w=us_{n-1}s_{n-2}\cdots s_k$ for some $1\leq k\leq n-1$ and $u\in\Sym_{n-1}$. Hence $$
T_w=T_uT_{n-1}T_{n-2}\cdots T_k .
$$
As a result,
\begin{equation}\label{eq:rel}
T_wL_n=T_uT_{n-1}L_nT_{n-2}\cdots T_k=T_u\Bigl(L_{n-1}T_{n-1}+1+(\xi-1)L_n\Bigr)T_{n-2}\cdots T_k.
\end{equation}
There are two possibilities:

\smallskip
{\it Case 1.} $0\leq a_n < \ell-1$. We have $$
\tau(L_n L_{1}^{a_1} \cdots L_{n}^{a_n} T_w)=\tau\bigl((L_{1}^{a_1} \cdots L_{n-1}^{a_{n-1}}T_u)L_{n}^{a_n+1} T_{n-1}T_{n-2}\cdots T_k\bigr)=0;
$$
%
If $u\neq 1$ or $k\leq n-2$, then by \eqref{eq:rel} $$\begin{aligned}
&\quad\,\tau(L_{1}^{a_1} \cdots L_{n}^{a_n} T_w L_n)\\
&=\tau(L_n^{a_n}L_{1}^{a_1} \cdots L_{n-1}^{a_{n-1}}T_uL_{n-1}T_{n-1}T_{n-2}\cdots T_k)+\tau(L_n^{a_n}L_{1}^{a_1} \cdots L_{n-1}^{a_{n-1}}T_uT_{n-2}\cdots T_k)+\\
&\qquad (\xi-1)\tau(L_n^{a_n+1}L_{1}^{a_1} \cdots L_{n-1}^{a_{n-1}}T_uT_{n-2}\cdots T_k)\\
&=0,
\end{aligned}
$$
where the last equality is obtained by applying the basis theorem (i.e., Lemma~\ref{stdBasis}) of $\HH_{n-1}$ to $L_{1}^{a_1} \cdots L_{n-1}^{a_{n-1}}T_uL_{n-1}$ and $L_{1}^{a_1} \cdots L_{n-1}^{a_{n-1}}T_uT_{n-2}\cdots T_k$. Assume $u=1$ and $k=n-1$. Then $w=s_{n-1}$, and $$ \begin{aligned}
\tau(L_{1}^{a_1} \cdots L_{n}^{a_n} T_w L_n)
&=\tau(L_n^{a_n}L_{1}^{a_1} \cdots L_{n-1}^{a_{n-1}})+(\xi-1)\tau(L_n^{a_n+1}L_{1}^{a_1} \cdots L_{n-1}^{a_{n-1}})\\
&=(1-\xi)^{n(\ell-1)-\sum_{j=1}^{n}a_j}+(\xi-1)(1-\xi)^{n(\ell-1)-\sum_{j=1}^{n}a_j-1}=0,
\end{aligned} $$
where the first equality follows from that $\tau (L_n^{a_n}L_{1}^{a_1} \cdots L_{n-1}^{a_{n-1}}L_{n-1}T_{n-1}) = 0$ which is obtained by applying the basis theorem (i.e. Lemma~\ref{stdBasis}) of $\HH_{n-1}$. This proves (\ref{commu3}) in Case 1..

\smallskip
{\it Case 2.} $a_n = \ell-1$. Then, by \eqref{eq:rel}, we have $$
\begin{aligned}
&\quad\, L_{1}^{a_1} \cdots L_{n}^{a_n} T_w L_n \\
&=L_{1}^{a_1} \cdots L_{n}^{\ell-1} T_u L_{n-1}T_{n-1}T_{n-2} \cdots T_{k}+(\xi -1) (L_{1}^{a_1} \cdots L_{n}^{\ell} T_u T_{n-2} \cdots T_{k}) + L_{1}^{a_1} \cdots L_{n}^{\ell-1} T_{u} T_{n-2} \cdots T_{k}\\
& \equal (\xi -1) (L_{1}^{a_1} \cdots L_{n}^{\ell} T_u T_{n-2} \cdots T_{k}) + L_{1}^{a_1} \cdots L_{n}^{\ell-1} T_{u} T_{n-2} \cdots T_{k},
\end{aligned} $$
 where the last equality is again obtained by applying the basis theorem (i.e. Lemma~\ref{stdBasis}) of $\HH_{n-1}$ to $L_{1}^{a_1} \cdots L_{n}^{\ell-1} T_u L_{n-1}$.

On the other hand, $$\begin{aligned}
      L_n L_{1}^{a_1} \cdots L_{n}^{a_n} T_w  & = L_{1}^{a_1} \cdots L_{n}^{\ell} T_{u} T_{n-1} \cdots T_{k} = L_{1}^{a_1} \cdots L_{n-1}^{a_{n-1}} T_{u} L_{n}^\ell T_{n-1} \cdots T_{k}.
     \end{aligned}
$$
Note that $L_n^\ell T_{n-1}=L_n^{\ell-1}\bigl((\xi-1)L_n+T_{n-1}L_{n-1}+1\bigr)=(\xi-1)L_n^\ell+L_n^{\ell-1}+L_n^{\ell-1}T_{n-1}L_{n-1}$. Thus in order to prove (\ref{commu3}) in Case 2, it suffices to show that
\begin{equation}\label{37}
\tau(L_{1}^{a_1} \cdots L_{n-1}^{a_{n-1}}T_{u}(L_{n}^{\ell-1}T_{n-1})L_{n-1} T_{n-2} \cdots T_{k})=0 .
\end{equation}

Note that $L_{1}^{a_1} \cdots L_{n-1}^{a_{n-1}} T_{u} \in \HH_{n-1,R}$, we can write it as an $R$-linear combination of the basis elements in $\{T_w L_1^{c_1}L_2^{c_2}\cdots L_{n-1}^{c_{n-1}}\bigm|w\in\Sym_{n-1}, 0\leq c_i<\ell,\forall\,1\leq i\leq n-1\}$. By the result we have proved in Step 1, to prove (\ref{37}) it suffices to show for any $0\leq a_1,\cdots,a_{n-1}<\ell$, $u\in \Sym_{n-1}$,
\begin{equation}\label{vanish}
\tau\bigl(L_{1}^{a_1} \cdots L_{n-1}^{a_{n-1}} L_{n}^{\ell-1}T_{n-1}L_{n-1} T_{n-2} \cdots T_{k} T_{u}\bigr)=0,
\end{equation}
To this end, we shall show that for any $1\leq a\leq\ell$ and any $0\leq t\leq\ell-a$, \begin{equation}\label{vanish2}
\tau(L_{1}^{a_1} \cdots L_{n-2}^{a_{n-2}}L_{n-1}^{\ell-a-t}L_{n}^{\ell-a}T_{n-1}L_{n-1}^{a} T_{n-2} \cdots T_{k}T_{u})=0 .
\end{equation}

We use induction on $t$. If $t=0$, then as $T_{n-1}$ commutes with $L_{n-1}L_n$, we can get that $$\begin{aligned}
&\quad\,\tau(L_{1}^{a_1} \cdots L_{n-1}^{\ell-a}(L_{n}^{\ell-a}T_{n-1}L_{n-1}^a) T_{n-2} \cdots T_{k}T_{u})\\
&=\tau(T_{n-1}L_{1}^{a_1} \cdots L_{n-1}^{\ell-a}(L_{n}^{\ell-a}L_{n-1}^a) T_{n-2} \cdots T_{k}T_{u})\\
&=\tau\bigl(L_{1}^{a_1} \cdots L_{n-1}^{\ell}T_{n-2} \cdots T_{k}T_{u}(L_{n}^{\ell-a}T_{n-1})\bigr)\quad\text{(by the result obtained in Step 1)}\\
&=0. \quad\text{(by Definition \ref{trace} and Lemma~\ref{stdBasis} for $\HH_{n-1}$)}
\end{aligned}
$$
Suppose that (\ref{vanish2}) holds for an integer $t\geq 0$. Now replacing $t$ with $t+1$, we have
$$\begin{aligned}
&\quad\,\tau(L_{1}^{a_1} \cdots L_{n-1}^{\ell-a-t-1}(L_{n}^{\ell-a}T_{n-1})L_{n-1}^{a} T_{n-2} \cdots T_{k}T_{u})\\
&=\tau(L_{1}^{a_1} \cdots L_{n-1}^{\ell-a-t-1}L_{n}^{\ell-a-1}\bigl((\xi-1)L_n+1+T_{n-1}L_{n-1})L_{n-1}^{a} T_{n-2} \cdots T_{k}T_{u})\\
&=\tau(L_{1}^{a_1} \cdots L_{n-1}^{\ell-a-t-1}L_n^{\ell-a-1}((\xi-1)L_n+1))L_{n-1}^{a} T_{n-2} \cdots T_{k}T_{u}),
\end{aligned}
$$
where the last step follows from induction hypothesis (with $a$ replaced with $a+1$).
Apply the basis theorem (i.e. Lemma~\ref{stdBasis}) for $\HH_{n-1}$ to $L_{1}^{a_1} \cdots L_{n-1}^{\ell-t-1}T_{n-2} \cdots T_{k}T_{u}$ and use definition~\ref{trace}, we have
$$
\tau(L_{1}^{a_1} \cdots L_{n-1}^{\ell-a-t-1}L_n^{\ell-a-1}((\xi-1)L_n+1))L_{n-1}^{a} T_{n-2} \cdots T_{k}T_{u})=0.
$$
as required. This proves (\ref{vanish2}) and hence (\ref{vanish}). Hence we complete the proof in the case $a_n=\ell-1$.

\medskip
\noindent
{\it Step 3.} Let $y,z\in\HH_{r,n}$. We claim that if $\tau(xy)=\tau(yx)$ and $\tau(xz)=\tau(zx)$ for any $x\in\HH_{n,R}$, then $\tau(x(yz))=\tau((yz)x)$ for any $x\in\HH_{n,R}$.

In fact, we have that $$
\tau(x(yz))=\tau((xy)z)=\tau(z(xy))=\tau((zx)y)=\tau(y(zx))=\tau((yz)x),
$$
as required.

Finally, since the Hecke algebra $\HH_{n,R}$ is generated by $T_1,T_2,\cdots,T_{n-1},L_n$, it follows from the results in Steps 1,2,3 that $\tau(xy)=\tau(yx), \forall\,x,y\in\HH_{n,R}$. In other words, the bilinear form induced from $\tau$ is a symmetric.
\end{proof}

Recall that for any $w\in\Sym_n$ and $0\leq c_1,\cdots,c_n<\ell$, we have $$
\tau^{\MM}(\mL^{c_1}_1 \cdots \mL^{c_n}_nT_w):=
\begin{cases}
1, &\text{if $w=1$ and $c_1=\cdots=c_n=0$ }\\
0, &\text{otherwise.}
\end{cases}
$$
where $\mL_1, \cdots, \mL_n$ are the Jucys-Murphy operators mentioned in Remark \ref{rem.L'}.

\begin{lem} \label{lem.taumtotau}
  Suppose $\xi\neq 1$. For any $w\in\Sym_n$ and $0\leq c_1,\cdots,c_n<\ell$, we have
  \begin{equation}\label{311}
    \tau^{\MM}(\mL^{c_1}_1 \cdots \mL_n^{c_n} T_{w})=(1-\xi)^{n(1-\ell)}\tau(\mL_1^{c_1} \cdots \mL_n^{c_n} T_{w}).
  \end{equation}
\end{lem}

\begin{proof}
  If $w \neq 1$, then $\tau(\mL^{c_1}_1 \cdots \mL^{c_n}_n T_{w})=\tau^{\MM} (\mL^{c_1}_1\cdots \mL^{c_n}_n T_{w})=0$. If $w=1$, by definition we have
  \begin{equation} \notag
    \tau^{\MM}(\mL_k)=0=(1-\xi)^{n(1-\ell)}\tau((\xi-1)L_k+1)=(1-\xi)^{n(1-\ell)}\tau(\mL_k).
  \end{equation}

We have
  $$\begin{aligned}
   \mL^{c_1}_1 \cdots \mL^{c_n}_n&=\bigl((\xi-1)L_1+1\bigr)^{c_1} \cdots \bigl((\xi-1)L_n+1\bigr)^{c_n} \\
    & = \Bigl(\sum_{k_1=0}^{c_1}\begin{pmatrix} c_1\\ k_1\end{pmatrix}(\xi-1)^{k_1}L_1^{k_1}\Bigr) \cdots \Bigl(\sum_{k_n=0}^{c_n}\begin{pmatrix} c_n\\ k_n\end{pmatrix}(\xi-1)^{k_n}L_n^{k_n}\Bigr).
  \end{aligned}$$
It follows that   $$\begin{aligned}
    \tau(\mL^{c_1}_1 \cdots \mL^{c_n}_n)&=\sum_{\substack{0\leq k_i\leq c_i,\\ 1\leq i\leq n}} \begin{pmatrix} c_1\\ k_1\end{pmatrix} \cdots\begin{pmatrix} c_n\\ k_n\end{pmatrix} (\xi-1)^{\sum_{j=1}^n k_j}(1-\xi)^{n(\ell-1)-\sum_{j=1}^n k_j}\\
    &=\sum_{\substack{0\leq k_i\leq c_i,\\ 1\leq i\leq n}} \begin{pmatrix} c_1\\ k_1\end{pmatrix} \cdots\begin{pmatrix} c_n\\ k_n\end{pmatrix} (-1)^{\sum_{j=1}^n k_j}(1-\xi)^{n(\ell-1)} \\
    & = (1-\xi)^{n(\ell-1)}\cdot \prod_{i=1}^n (1+(-1))^{c_i} \\
    & = \begin{cases} (1-\xi)^{n(\ell-1)} & \text{if }c_1=\dots=c_n=0, \\ 0 & \text{otherwise,}  \end{cases} \\
    & = (1-\xi)^{n(\ell-1)}\tau^{\MM} (\mL^{c_1}_1 \cdots \mL^{c_n}_n).
  \end{aligned}$$
This completes the proof of (\ref{311}).
\end{proof}

Recall the definition of $\mathbf{n}(\blam)$ in Definition (\ref{mnlam}).

\begin{cor}\label{invertible0} Let $\blam\in\Parts[\ell,n]$. We have $$
\tau_{R}\bigl(\fm_\blam^R T_{w_{\blam}} \fn_{\blam}^R T_{w_{\blam'}}\bigr)=(-1)^{\mathbf{n}(\blam)}\xi^{\mathbf{n}(\blam)+\ell(w_{\blam})} \prod_{s=1}^{\ell}\Bigl(Q_s(\xi-1)+1\Bigr)^{n-|\lam^{(s)}|}.
$$
\end{cor}

\begin{proof} This follows from \cite[Theorem 5.9]{Ma} and \cite[Theorem 4.7]{Zh} by taking into account of the difference between our $\fm_\blam, \fn_\blam$ with that in \cite{Ma}.
\end{proof}

Let $\O$ be an integral domain and $A$ an $\O$-algebra which is a  free $\O$-module of finite rank. Recall that $A$ is called a symmetric $\O$-algebra if there is a trace function $\tr: A\rightarrow\O$ (i.e., $\tr(hh')=\tr(h'h), \forall\,h,h'\in A$) such that the determinant of the matrix $(\tr(bb'))_{b,b'\in\mathcal{B}}$ is a unit in $\O$ for some (and hence every) $\O$-basis $\mathcal{B}$ of $A$.

The following theorem is the first main result of this paper.

\begin{thm}\label{mainthm11} Let $R$ be an integral domain.
For any $w\in\Sym_n$ and integers $0\leq c_1,c_2,\cdots,c_n<\ell$, we define $$
\tau_R(L_1^{c_1}\cdots L_n^{c_n}T_w):=\begin{cases} (1-\xi)^{n(\ell-1)-\sum_{i=1}^n c_i}, &\text{if $w=1$;}\\
        0, &\text{otherwise.}
\end{cases}
$$
We extend the map $\tau_R$ linearly to an $R$-linear function on $\HH_{n,R}$. If $Q_r(\xi-1)+1\in R^\times$ for any $1\leq r\leq\ell$, then  $\tau_R$ is a non-degenerate symmetrizing form on $\HH_{n,R}$. In other words, $\HH_{n,R}$ is a symmetric $R$-algebra.
\end{thm}

\begin{proof}
Theorem \ref{mainthm01} has proved that $\tau$ induces a symmetric bilinear form on $\HH_{n,R}$. It remains to show that there are two $R$-bases $\{b_i\}, \{b'_j\}$ of $\HH_{n,R}$ such that the determinant of the matrix $(\tr(b_ib'_j))_{i,j}$ is a unit in $R$. To this end, we consider the cellular basis $\{\fm_{\s\t}^R\}$ and the dual cellular basis $\{\fn_{\v\u}^R\}$ of $\HH_{n,R}$. It suffices to show that \begin{enumerate}
\item $\tau_R\bigl(\fm_{\s\t}^R\fn_{\t\s}^R\bigr)\in R^\times$; and
\item $\tau_R\bigl(\fm_{\s\t}^R\fn_{\v\u}^R\bigr)\neq 0$ only if $(\u,\v)\unrhd(\s,\t)$.
\end{enumerate}

Let $(\s,\t)\in\Std^2(\blam)$. For Part a), we have that $$\begin{aligned}
\tau_R\bigl(\fm_{\s\t}^R\fn_{\t\s}^R\bigr)&=\tau_R\bigl(T_{d(\s)}^*\fm_{\blam}^R T_{d(\t)}T_{d'(\t)}^*\fn_{\blam}^RT_{d'(\s)}\bigr)
=\tau_R\bigl(\fm_{\blam}^R(T_{d(\t)}T_{d'(\t)}^*)\fn_{\blam}^R(T_{d'(\s)}T_{d(\s)}^*)\bigr)\\
&=\tau_{R}\bigl(\fm_\blam^R T_{w_{\blam}} \fn_{\blam}^R T_{w_{\blam'}}\bigr)\in R^\times ,  \quad \text{(by Corollary \ref{invertible0})}
\end{aligned}
$$
as required.

It remains to prove Part b). Let $K$ be the fraction field of $R$ and $\O_K,\K$ be as in the setting before Definition~\ref{dfn:Ft}. Suppose that $(\u,\v)\ntrianglerighteq(\s,\t)$. We want to prove $\tau_R\bigl(\fm_{\s\t}^R\fn_{\v\u}^R\bigr)=0$. It is enough to show that $\tau_\O\bigl(\fm_{\s\t}^{\O}\fn_{\v\u}^{\O}\bigr)=0$ for $\O = \O_K$.

By Lemmas \ref{fsemi} and \ref{gsemi}, $$\begin{aligned}
\fm_{\s\t}^{\O}&=f_{\s\t}+\sum_{(\a,\b)\rhd(\s,\t)}r_{\a\b}^{\s\t}f_{\a\b},\\
\fn_{\v\u}^{\O}&=g_{\v\u}+\sum_{(\b,\a)\lhd(\v,\u)}\check{r}_{\b\a}^{\v\u}g_{\b\a},
\end{aligned}
$$
where $r_{\a\b}^{\s\t}, \check{r}_{\b\a}^{\v\u}\in\K$ for any pair $(\b,\a)$. By Theorem \ref{mainthm01}, $\tau(xy)=\tau(yx), \forall\,x,y\in\HH_{n}$.
Note that (cf. \cite{HW}) $g_{\s\t}=c_{\s\t}f_{\s\t}$ for some $c_{\s\t\in \K^\times}$. Combining these properties with Lemma \ref{fsemi}, we can deduce that $\tau_\K(f_{\a\b}g_{\q\p})\neq 0$ only if $\b=\q$ and $\a=\p$. Thus $\tau_R\bigl(\fm_{\s\t}^R\fn_{\v\u}^R\bigr)= 0$ since we cannot find $(\a,\b)$ such that $(\s,\t) \lhd (\a,\b) \lhd (\u,\v)$ when $(\u,\v)\ntrianglerighteq(\s,\t)$.
This completes the proof of Part b) and hence the whole theorem.
\end{proof}

\begin{dfn} Let $\O\in\{K,\O_K,\K\}$. For any $\blam\in\Parts[\ell,n]$, we define $$
z^\O_\blam:=\fm^\O_\blam T^\O_{w_{\blam}}\fn^\O_{\blam},\,\,\, z_\blam:=z_\blam^K .
$$
\end{dfn}

The element $z_\blam T_{w_{\blam'}}$ is actually a quasi-idempotent (i.e., $z_\blam T_{w_{\blam'}})^2=cz_\blam T_{w_{\blam'}}$ for some $c\in K$), which plays a key role in the study of symmetrizing form on $\HH_{n,K}$ in \cite{Ma}.

\begin{lem}\label{keylem21} Let $\blam\in\Parts[\ell,n]$. Then \begin{enumerate}
\item $z^{\O_K}_\blam T^{\O_K}_{w_{\blam'}}=\gamma_{\t_\blam}\check{\gamma}_{\t^{\blam'}}f_{\tlam,\tlam}/\gamma_{\tlam}+\sum_{\tlam\neq\t\in\Std(\blam)}a_\t f_{\tlam,\t}$, where $a_\t\in\K$ for each $\t$;
\item Let $\O\in\{K,\O_K,\K\}$. Suppose that $Q_r(\xi-1)+1\in \O^\times$ for any $1\leq r\leq\ell$. Then $0\neq z^\O_\blam T^\O_{w_{\blam'}}\notin[\HH_{n,\O},\HH_{n,\O}]$.
\end{enumerate}
\end{lem}

\begin{proof} By the same argument as \cite[Remark 3.6]{Ma}, we can get that $g_{\t_\blam\t_\blam}=\check{\gamma}_{\t^{\blam'}}f_{\t_\blam\t_\blam}/\gamma_{\t_\blam}$. Then Part 1) follows from the same proof as \cite[Proposition 3.13]{Ma} and the proof of \cite[Proposition 4.4]{Ma}. Note that the statement in \cite[Proposition 4.4]{Ma} missed the terms $\sum_{\tlam\neq\t\in\Std(\blam)}a_\t f_{\tlam,\t}$. Part 2) follows from Corollary \ref{invertible0}.
\end{proof}

For any $\blam\in\Parts[\ell,n]$, and $\gamma=(r,c,l)\in[\blam]$, we define
\begin{equation}\label{res}
\res(\gamma):=\xi^{j-i}{Q}_c+(1+{\xi}+\cdots+{\xi}^{j-i-1}) .
\end{equation}

\begin{dfn} Let $\blam\in\Parts[\ell,n]$. If the multi-set $\{\res(\gamma)|\gamma\in[\blam]\}$ is equal to the multi-set $\{\res(\gamma)|\gamma\in[\bmu]\}$, then we say that $\blam,\bmu$ are residue equivalent and denote $\blam \overset{\res}{\sim} \bmu$.
\end{dfn}

For each $\blam\in\Parts[\ell,n]$, we use $S(\blam)$ to denote the corresponding (left) cell module defined via the cellular bases $\{\fm_{\s\t}\}$ and call it the Specht module labelled by $\blam$. Using \cite[Lemma 2.9]{DR} we can deduce that $S(\blam)\cong\HH_{n,K}z_\blam$. By \cite{DM}, \cite{LM:AKblocks} and \cite{Brundan:degenCentre}, we know that $\blam,\bmu$ are residue equivalent if and only if the corresponding Specht modules $S(\blam), S(\bmu)$ of $\HH_{n,K}$ are in the same block.

\begin{prop}\label{keylem22} Suppose that $Q_r(\xi-1)+1\in K^\times$ for any $1\leq r\leq\ell$. For any $\blam,\bmu\in\Parts[\ell,n]$, the element $z_\blam T_{w_{\blam'}}+[\HH_{n,K},\HH_{n,K}]$ and $z_\bmu T_{w_{\bmu'}}+[\HH_{n,K},\HH_{n,K}]$ are $K$-linearly dependent if and only if $\blam$ and $\bmu$ are residue equivalent.
\end{prop}

\begin{proof}

Let $e(1)$ (resp., $e(2)$) be the central primitive idempotent corresponding to the Specht module $S(\blam)$ (resp. $S(\bmu)$).

Suppose that $\blam$ and $\bmu$ are not residue equivalent. Then $S(\blam), S(\bmu)$ are not in the same block and thus $e(1)e(2)=0=e(2)e(1)$. Suppose that $$
c_1z_\blam T_{w_{\blam'}}+c_2z_\bmu T_{w_{\bmu'}}\in [\HH_{n,K},\HH_{n,K}],\,\,\text{where}\,\,c_1,c_2\in K .
$$
Note that $S(\blam)\cong\HH_{n,K}z_\blam$, $S(\bmu)\cong\HH_{n,K}z_\bmu$. Multiplying $e(1)$ and $e(2)$ on the above equality we can deduce that $c_1=c_2=0$. Hence the element $z_\blam T_{w_{\blam'}}+[\HH_{n,K},\HH_{n,K}]$ and $z_\bmu T_{w_{\bmu'}}+[\HH_{n,K},\HH_{n,K}]$ are $K$-linearly independent.

Now, suppose that $\blam$ and $\bmu$ are not residue equivalent.
Applying Corollary \ref{invertible0} and our assmption on $\xi, Q_r$, we see that $\tau_K(z_\blam T_{w_{\blam'}})\neq 0\neq\tau_K(z_\bmu T_{w_{\bmu'}})$, we can find  $c_{\blam,\bmu}\in K^\times$ such that $\tau_K\bigl(z_\bmu T_{w_{\bmu'}}-c_{\blam,\bmu} z_\blam T_{w_\blam}\bigr)=0$. We claim that
$z_\bmu T_{w_{\bmu'}}-c_{\blam,\bmu} z_\blam T_{w_\blam}\in [\HH_{n,K},\HH_{n,K}]$.

Suppose that this is not the case, i.e., $z_\bmu T_{w_{\bmu'}}-c_{\blam,\bmu} z_\blam T_{w_\blam}\notin [\HH_{n,K},\HH_{n,K}]$. Since $\HH_{n,K}$ is symmetric, we have $Z(\HH_{n,K})\cong (\Tr(\HH_{n,K}))^*=(\HH_{n,K}/[\HH_{n,K},\HH_{n,K}])^*$. Thus we can find a center element $z\in Z(\HH_{n,K})$ such that $$
\tau_K\Bigl(zz_\bmu T_{w_{\bmu'}}-c_{\blam,\bmu}zz_\blam T_{w_\blam}\Bigr)\neq 0.
$$

On the other hand, the center element $z$ acts on each Specht module $S(\blam)$ by a scalar $c_\blam$ because $\End_{\HH_{n,K}}(\HH_{n,K}z_\blam)=K$. In particular, $ zz_\blam=c_\blam z_\blam, zz_\bmu=c_\bmu z_\bmu$. Since $\HH_{n,K} z_\blam\cong S(\blam)$ and $\HH_{n,K}z_\bmu\cong S(\bmu)$ are in the same block, it follows that $c_\blam=c_\bmu$. Hence  $\tau\bigl(z(z_\bmu T_{w_{\bmu'}}-c_{\blam,\bmu} z_\blam T_{w_\blam})\bigr)=0$, which is a contradiction. This completes the proof of the proposition.
\end{proof}

\bigskip

\section{Dual bases of cyclotomic Hecke algebras}

The purpose of this section is to construct an explicit pair of dual bases for the  cyclotomic Hecke algebra of type $A$ with respect to the symmetrizing form $\tau$ introduced in the last section.

Recall the constants $\gamma_\s, \check{\gamma}_\t$ introduced in Lemmas \ref{fsemi}, \ref{gsemi}. Recall that $\HH_{n,\K}$ is the cyclotomic Hecke algebra of type $A$, defined over $\K$, with Hecke parameter and cyclotomic parameters given by (\ref{OHecke}).

\begin{lem}\text{\rm (\cite[Theorem 5.9, Proposition 6.1]{Ma}, \cite[Lemma 5.3]{Zh})}\label{tr0} For any $\blam\in\Parts[\ell,n]$, and any $\s,\t\in\Std(\blam)$, we have $$
\tau_{\K}(f_{\s\t})=\delta_{\s\t}(-1)^{\mathbf{n}(\blam)}\gamma_{\s}(\gamma_{\t_\blam}\check{\gamma}_{\t^{\blam'}})^{-1}\hat{\xi}^{{\rm n}(\blam)+\ell(w_\lam))}\prod_{s=1}^{\ell}\Bigl(\hat{Q}_s(\hat{\xi}-1)+1\Bigr)^{n-|\lam^{(s)}|}.
$$
\end{lem}

\begin{proof} The lemma follows from \cite[Theorem 5.9, Proposition 6.1]{Ma} and \cite[Lemma 5.3]{Zh} by taking into account the scalars by which our $\fm_\blam, \fn_\blam, \gamma_\s, \check{\gamma}_\t$ differ with the notations $m_\blam, n_{\blam'}, \gamma_\s, \gamma'_\t$ in \cite{Ma} when $t\neq 1$.
\end{proof}

\begin{lem} For any $\blam\in\Parts[\ell,n]$ and $\s,\t\in\Std(\blam)$, we have \begin{equation}\label{tr01}
\tau_\K(f_{\t\s}g_{\s\t})=\tau_{\O_K}(\fm^{\O_K}_{\t\s}\fn^{\O_K}_{\s\t})=(-1)^{\mathbf{n}(\blam)}\hat{\xi}^{\mathbf{n}(\blam)+\ell(w_{\blam})} \prod_{s=1}^{\ell}\Bigl(\hat{Q}_s(\hat{\xi}-1)+1\Bigr)^{n-|\lam^{(s)}|}.
\end{equation}
\end{lem}

\begin{proof} For any $\blam\in\Parts[\ell,n]$ and $\s,\t\in\Std(\blam)$, we have $$
d'(\t)d(\t)^{-1}=w_{\blam'}=w_{\blam}^{-1},\,\quad \ell(d'(\t)+\ell(d(\t)^{-1})=\ell(w_{\blam'}).
$$
Combining this with Theorems \ref{fsemi} and \ref{gsemi} we get that $$\begin{aligned}
\tau_{\O_K}(f_{\t\s}g_{\s\t})&=\tau_{\O_K}(\fm^{\O_K}_{\t\s}\fn^{\O_K}_{\s\t})=\tau_{{\O_K}}\Bigl(\fm_\blam (T_{d(\s)}T_{d'(\s)}^*)\fn_\blam \bigl(T_{d'(\t)}T_{d(\t)}^*\bigr)\Bigr)=\tau_{{\O_K}}\bigl(\fm_\blam T_{w_{\blam}} \fn_{\blam} T_{w_{\blam'}}\bigr)\\
&=(-1)^{\mathbf{n}(\blam)}\hat{\xi}^{{\rm n}(\blam)+\ell(w_\lam)}\prod_{s=1}^{\ell}\Bigl(\hat{Q}_s(\hat{\xi}-1)+1\Bigr)^{n-|\lam^{(s)}|}
\end{aligned}
$$
This proves the lemma.
\end{proof}

To simplify the notations, we define for any $\blam\in\Parts[\ell,n]$, $$
r_\blam(x):=(-1)^{\mathbf{n}(\blam)}(x+\xi)^{-\mathbf{n}(\blam)-\ell(w_{\blam})} \prod_{s=1}^{\ell}\Bigl(\hat{Q}_s(x+\xi-1)+1\Bigr)^{|\lam^{(s)}|-n}.
$$

Recall that $\O_K:=K[x]_{(x)}$. Let $\hO_K$ be the completion of $\O_K$ with respect to its unique maximal ideal generated by $x$. Let $\hat{K}:=K((x))$ be the fraction field of $\hat{\O}_{K}$. Then $\hat{\O}_K$ has $K$ as its residue field.  As before, we can identify $f_{\s\t}, g_{\s\t}$ with $1_{\hat{K}}\otimes_{\K}f_{\s\t}, 1_{\hat{K}}\otimes_{\K}g_{\s\t},$ respectively. Note that $\{r_\blam(x)\fm_{\s\t}|\s,\t\in\Std(\blam),\blam\in\P_{\ell,n}\}$ is again a cellular basis of $\HH_{n,\O}$, and $\{r_\blam(x)f_{\s\t}|\s,\t\in\Std(\blam),\blam\in\P_{\ell,n}\}$ is the associated seminormal basis of $\HH_{n,\K}$.

\begin{lem}\text{(\cite[Theorem 6.2]{HuMathas:SeminormalQuiver})}\label{2DBases} For each $\blam\in\Parts[\ell,n]$ and $\s,\t\in\Std(\blam)$, there is a unique element $B_{\s\t}^{\hat{O}_K}\in\HH_{n,\hat{\O}_K}$ such that $$
B_{\s\t}^{\hat{O}_K}=r_\blam(x)f_{\s\t}+\sum_{\substack{\u,\v\in\Std(\bmu),\bmu\in\Parts[\ell,n]\\ (\u,\v)\rhd(\s,\t)}}r^{\s\t}_{\u\v}r_\bmu(x)f_{\u\v},
$$
where $r^{\s\t}_{\u\v}\in x^{-1}K[x^{-1}]$ for each pair $(\u,\v)$. Moreover, the set $\{B_{\s\t}^{\hat{O}_K}|\s,\t\in\Std(\blam),\blam\in\Parts[\ell,n]\}$ form an $\hat{\O}_K$-basis of $\HH_{n,\hat{\O}_K}$.
Similarly, there is a unique element $\check{B}_{\s\t}^{\hat{O}_K}\in\HH_{n,\hat{\O}_K}$ such that $$
\check{B}_{\s\t}^{\hat{O}_K}=g_{\s\t}+\sum_{\substack{\u,\v\in\Std(\bmu),\bmu\in\Parts[\ell,n] \\ (\u,\v)\lhd(\s,\t)}}\check{r}^{\s\t}_{\u\v}g_{\u\v},
$$
where $\check{r}^{\s\t}_{\u\v}\in x^{-1}K[x^{-1}]$ for each pair $(\u,\v)$. Moreover, the set $\{\check{B}_{\s\t}^{\hat{O}_K}|\s,\t\in\Std(\blam),\blam\in\Parts[\ell,n]\}$ form an $\hat{\O}_K$-basis of $\HH_{n,\hat{\O}_K}$. \end{lem}

\begin{proof} The first part follows from \cite[Theorem 6.2]{HuMathas:SeminormalQuiver}. Replacing $\{r_\blam(x)f_{\s\t}\}, \{r_\blam(x)\fm^{\hat{\O}_K}_{\s\t}\}$ with $\{g_{\s\t}\}, \{\fn_{\s\t}^{\hat{\O}_K}\}$ respectively, we can prove the second part of the lemma in a similar way as the proof of   \cite[Theorem 6.2]{HuMathas:SeminormalQuiver}.
\end{proof}

\begin{dfn} We define \begin{equation}\label{DBasis}
B_{\s\t}^K:=1_K\otimes_{\hat{O}_K}B_{\s\t}^{\hat{O}_K},\quad \check{B}_{\s\t}^K:=1_K\otimes_{\hat{O}_K}\check{B}_{\s\t}^{\hat{O}_K},
\end{equation}
For $R\in\{\hO_K, K\}$, we call $\{B_{\s\t}^R|\s,\t\in\Std(\blam),\blam\in\Parts[\ell,n]\}$ the {\it distinguished $R$-basis} of $\HH_{n,R}$, and $\{\check{B}^R_{\s\t}|\s,\t\in\Std(\blam),\blam\in\Parts[\ell,n]\}$ the {\it dual distinguished $R$-basis} of $\HH_{n,R}$.
\end{dfn}

The following theorem is the second main result of this paper.

\begin{thm}\label{mainthm22} Suppose that \begin{equation}\label{invert1}
Q_r(\xi-1)+1\in K^\times,\quad\forall\, 1\leq r\leq\ell .
\end{equation}

1) The $\hat{\O}_K$-algebra $\HH_{n,\hat{\O}_K}$ is a symmetric $\hat{\O}_K$-algebra with symmetrizing form $\tau_{\hO_K}$, and the two $\hat{O}_K$-bases $\{B_{\s\t}^{\hat{\O}_K}|\s,\t\in\Std(\blam),\blam\in\Parts[\ell,n]\}$ and $$
\Biggl\{\check{B}_{\t\s}^{\hat{\O}_K}\Biggm|\s,\t\in\Std(\blam),\blam\in\Parts[\ell,n]\Biggr\}
$$ are dual to each other under $\tau_{\hat{\O}_K}$.

2) The $K$-algebra $\HH_{n,K}$ is a symmetric $K$-algebra with symmetrizing form $\tau_K$, and the two $K$-bases $\{B_{\s\t}^{K}|\s,\t\in\Std(\blam),\blam\in\Parts[\ell,n]\}$ and $$
\Biggl\{\check{B}_{\t\s}^{K}\Biggm|\s,\t\in\Std(\blam),\blam\in\Parts[\ell,n]\Biggr\}
$$ are dual to each other under $\tau_{K}$.
\end{thm}

\begin{proof} Since $\tau_K=1_K\otimes_{\hO_K}\tau_{\hO_K}$, Part 2) is consequence of Part 1). It remains to prove Part 1) of the theorem.

Recall that for any $\blam\in\Parts[\ell,n]$, $$
r_\blam(x)=(-1)^{\mathbf{n}(\blam)}(x+\xi)^{-\mathbf{n}(\blam)-\ell(w_{\blam})} \prod_{s=1}^{\ell}\Bigl(\hat{Q}_s(x+\xi-1)+1\Bigr)^{|\lam^{(s)}|-n}.
$$
The assumption (\ref{invert1}) implies that $r_\blam(x)\in\O_K^\times$.
By Lemma \ref{tr0}, for any $\blam,\bmu\in\Parts[\ell,n]$, $\s,\t\in\Std(\blam)$ and $\u,\v\in\Std(\bmu)$, we have that $$
\tau_{\hat{K}}\bigl(r_\blam(x)f_{\s\t}g_{\u\v}\bigr)=\tau_{\hat{K}}\bigl(r_\blam(x)g_{\u\v}f_{\s\t}\bigr)=\delta_{\t\u}\delta_{\s\v}.
$$
Now Part a) of the proposition follows from the above equality, Lemma~\ref{2DBases} and the fact that ${\hat{\O}}_K\cap x^{-1}K[x^{-1}]=\{0\}$. This completes the proof of the theorem.
\end{proof}

\bigskip

\section{Cocenters of the cyclotomic Schur algebras}

In this section, we shall study the cyclotomic Schur algebra corresponding to the cyclotomic Hecke algebra $\HH_{n,R}$ of type $A$. The main result is Theorem \ref{mainthm33}, which gives an explicit integral basis for the cocenter of the cyclotomic Schur algebra.

Let $R$ be an integral domain, $\xi\in R^\times$ and $Q_1,\cdots,Q_\ell\in R$. Let $\HH_{n,R}$ be the cyclotomic Hecke algebra of type $A$, defined over $R$, with Hecke parameter $\xi$ and cyclotomic parameters $Q_1,\cdots,Q_\ell$.
Let $\mC$ is a subset of the set of all multicompositions of $n$ which have $\ell$ components such that if $\bmu\in\mC$
and $\bnu\in\Parts[\ell,n]$ with $\bnu\rhd\bmu$ then $\bnu\in\mC$. We let $\mC^+$ be the set of multipartitions in $\mC$. By definition \ref{CSchur}, $$
S_{n,R}:=\End_{\HH_{n,R} }\Bigl(\bigoplus_{\bmu\in\mC}\fm_\bmu^{R}\HH_{n,R} \Bigr).
$$

\begin{dfn}\label{dfn:cenopr} Let $\blam\in\mC$. For each $z\in Z(\HH_{n,R})$, we use $\phi_{\blam\blam}^z$ to denote the $\HH_{n,R}$-endomorphism of $\fm_\blam^{R}\HH_{n,\hO}$ given by left multiplication with $z$. We extend $\phi_{\blam\blam}^z$ to an element of $S_{n,R}$ by defining $\phi_{\blam\blam}^z$ to be zero on $\fm_{\brho}^{R}\HH_{n,R}$ whenever $\blam\neq\brho\in\mC$.
\end{dfn}

The following lemma is usually referred as Schur-Weyl duality or double centralizer property between $S_{n,R}$ and $\HH_{n,R}$.

\begin{lem}\text{\rm (\cite[Theorem 5.3]{Ma04})}\label{SW} We have $$
S_{n,R}=\End_{\HH_{n,R}}\Bigl(\bigoplus_{\bmu\in\mC}\fm_{\bmu}^R\HH_{n,R}\Bigr),\quad (\HH_{n,R})^{\rm{op}}=\End_{S_{n,R}}\Bigl(\bigoplus_{\bmu\in\mC}\fm_{\bmu}^R\HH_{n,R}\Bigr).
$$
\end{lem}

\begin{cor}\label{keycor22} There is an $R$-algebra isomorphism $\theta_R: Z(\HH_{n,R})\cong Z(S_{n,R})$ such that $$
\theta_R(z)=\sum_{\bmu\in\mC}\phi_{\bmu\bmu}^z,\,\,\,\forall\,z\in Z(\HH_{n,R}).
$$
\end{cor}

\begin{proof} This follows from Lemma \ref{SW} because both $Z(\HH_{n,R})$ and $Z(S_{n,R})$ are identified with those $R$-linear endomorphism of $\bigoplus_{\bmu\in\mC}\fm_{\bmu}^R\HH_{n,R}$ which commutes with both the right action of $\HH_{n,R}$ and the left action of $S_{n,R}$.
\end{proof}

For any $R$-algebra $A$, we define the cocenter $\Tr(A)$ of $A$ to be the $R$-module $\Tr(A):=A/[A,A]$, where $[A,A]$ is the $R$-submodule of $A$ generated by all commutators of the form $xy-yx$ for $x,y\in A$. Note that $\Tr(A)$ is the $0$-th Hochschild homology ${\rm H\!H}_0(A)$ of $A$. In order to study the cocenter of $S_{n,R}$, we need to recall the seminormal basis theory for the cyclotomic Schur algebra $S_{n}$ which was developed in \cite{MathasSeminormal} in a more general setup.

\begin{dfn}\text{(\cite[Lemma 3.3]{JM})}  Let $1\leq i\leq n$ and $k\in\{1,2,\cdots,\ell\}$. Let $\bmu\in\mC$ and $\t^\bmu=(\t^{\mu^{(1)}},\cdots,\t^{\mu^{(\ell)}})$.
Assume $a,a+1,\cdots,z$ are the entries in row $i$ of $\t^{\mu^{(k)}}$. We set $$
L_{i,k}^{\bmu}:=L_a+L_{a+1}+\cdots+L_z,
$$
and define $\mathcal{L}_{i,k}^{\bmu}$ to be the element in $\Hom_{\HH_{n}}\bigl(\fm_{\bmu}\HH_{n},\fm_{\bmu}\HH_{n}\bigr)$ given by $$
\mathcal{L}_{i,k}^{\bmu}(\fm_{\bmu} h):=L_{i,k}^{\bmu}\fm_{\bmu} h=(L_a+L_{a+1}+\cdots+L_z)\fm_{\bmu} h,\,\,\,\forall\,h\in\HH_{n}.
$$
\end{dfn}
Note that our definition above is slightly different with \cite[Lemma 3.3]{JM} as the Jucys-Murphy operators used in \cite[Lemma 3.3]{JM} are $\mathcal{L}_1,\cdots,\mathcal{L}_n$ in our notations. However, this difference will not affect our following discussion (as we have changed the definitions of $\cont(\gamma), \cont_{\bT}(i,k)$ accordingly). We call $\{\mathcal{L}_{i,k}^{\bmu}\}$ the {\bf Jucys -Murphy operators} of $S_{n,R}$. By \cite[Corollary 2.3]{DJ2}, we know that $L_a+L_{a+1}+\cdots+L_z$ commutes with any element in the Iwahori-Hecke algebra $\HH_{\xi}(\Sym_\bmu)$ associated to the parabolic subgroup $\Sym_{\bmu}$ and hence
$\mL_{i,k}^{\bmu}$ is indeed a right $\HH_{n,R}$-module endomorphism of $\fm_{\bmu}\HH_{n,R}$.

Let $K$ be a field and $x$ be an indeterminant over $K$. Set $\O_K:=K[x]_{(x)}$. For $R\in\{\O_K,\K\}$, let $\HH_{n,R}$ be the cyclotomic Hecke algebra with Hecke parameter $\hat{\xi}$ and cyclotomic parameters $\hat{Q}_1,\cdots,\hat{Q}_\ell$ given by (\ref{OHecke}). In particular, there is a natural embedding $S_{n,\O_K}\hookrightarrow S_{n,\K}$. For any $k,m\in\mathbb{N}$ with $k\leq m$, we set $[k,m]:=\{k,k+1,\cdots,m\}$.

\begin{dfn}\text{(\cite[Definition 3.4(2)]{JM})} Let $\bT\in\mT(\blam,\bmu)$ and $(i,k)\in [1,n]\times [1,\ell]$. We define $$
\cont_{\bT}(i,k):=\sum_{\substack{\gamma\in[\blam]\\ \bT(\gamma)=(i,k)}}\cont(\gamma) ,
$$
where $\cont(\gamma)$ is defined as (\ref{content}).
It is easy to check that
\begin{equation}\label{invertible} \cont_{\bT}(i,k)\in\K^\times,\quad\,\forall\, 1\leq i\leq n,\,\, 1\leq k\leq\ell.
\end{equation}
\end{dfn}

The following lemma can be proved in a similar argument as that used in the proof of \cite[Lemma 3.12]{JM}, see also \cite[2.5]{Ma} and \cite[Lemma 6.5]{AMR}.

\begin{lem} Let $\bS\in\mT(\blam,\bmu)$ and $\bT\in\mT(\brho,\bnu)$ be distinct semistandard tableaux. Then $$
\cont_{\bS}(i,k)\neq\cont_{\bT}(i,k),
$$
for some $(i,k)\in [1,n]\times [1,\ell]$.
\end{lem}

\begin{lem}\text{(\cite[Theorem 3.10]{JM})} Let $(i,k)\in  [1,n]\times [1,\ell]$. Let $\blam\in\Parts[\ell,n], \brho,\bmu,\bnu\in\mC$ and
$\bS\in\mT(\blam,\bnu), \bT\in\mT(\blam,\bmu)$. Then $$
\varphi_{\bS\bT}\mL_{i,k}^{\brho}=\begin{cases}\cont_{\bT}(i,k)\varphi_{\bS\bT}+\sum_{\bT\lhd\bU\in\mT(\blam,\bmu)}a_{\bU}\varphi_{\bS\bU}\,\, \pmod{(S_{n})^{\rhd\blam}}, &\text{if $\brho=\bmu$;}\\
0, &\text{if $\brho\neq\bmu$,}\end{cases}
$$
where $a_{\bU}\in R$ for each $\bU$.
\end{lem}
In the notation of the above lemma, we set $$\begin{aligned}
r_{\bT}(\brho,i,k)&:=\begin{cases} \cont_\bT(i,k), &\text{if $\brho=\bmu$;}\\ 0, &\text{if $\brho\neq\bmu$.}
\end{cases}\\
\mathcal{C}(\brho,i,k)&:=\bigl\{r_{\bT}(\brho,i,k)\bigm| \bT\in\mT(\blam,\bmu),\blam\in\Parts[\ell,n],\bmu\in\mC\bigr\}.
\end{aligned}
$$

\begin{dfn}\text{(\cite[Definition 3.1]{MathasSeminormal})} Let $\bT\in\mT(\blam,\bmu)$. Define $$
\mathcal{F}_{\bT}:=\prod_{\substack{(i,k)\in [1,n]\times [1,\ell]\\ \brho\in\mC}}\prod_{\substack{r(\brho,i,k)\in \mathcal{C}(\brho,i,k) \\
r(\brho,i,k)\neq r_\bT(\brho,i,k)}}
\frac{\mathcal{L}_{i,k}^{\brho}-r(\brho,i,k)}{r_{\bT}(\brho,i,k)-r(\brho,i,k)}\in S_{n,\K} .
$$
If $\bS\in\mT(\blam,\bnu)$, then we define $$
\mathcal{F}_{\bS\bT}:=\mathcal{F}_{\bS}\varphi_{\bS\bT}^{\O_K}\mathcal{F}_{\bT}=\mathcal{F}_{\bS}\varphi_{\bS\bT}^{\K}\mathcal{F}_{\bT} .
$$
\end{dfn}
In particular, by (\ref{invertible}), we have $$
\mathcal{F}_{\bS\bT}(\fm_{\bmu}\HH_{n,\K})\subseteq \fm_{\bnu}\HH_{n,\K},\quad\, \mathcal{F}_{\bS\bT}(\fm_{\brho}\HH_{n,\K})=0,\,\,\forall\,\brho\neq\bmu .
$$

\begin{lem}\label{Ftsemi2} Assume that $\omega:=(\emptyset,\cdots,\emptyset,(1^n))\in\mC^+$. Let $\blam\in\Parts[\ell,n]$. Then for any $\s,\t\in\Std(\blam)=\mathcal{T}_0(\blam,\omega)$, $\mathcal{F}_{\t}$ is the left multiplication of $\HH_{n,\K} $ with $F_\t$, and $\mathcal{F}_{\s\t}$ is the left multiplication of $\HH_{n,\K} $ with $f_{\s\t}$.
\end{lem}

\begin{proof} The second part of the lemma follows from Lemma \ref{varphim1} and the first part of the lemma. For the first part, using \cite[Proposition 2.6, Theorem 2.15]{Ma} and \cite[Proposition 3.4, Theorem 3.16]{MathasSeminormal}, we know that $\mathcal{F}_{t}$ acts on $\HH_{n,\K} $ via the left multiplication with the following element $$
\widetilde{F}_\t:=\prod_{\substack{(i,k)\in [1,n]\times [1,\ell]}}\prod_{\substack{r(\omega,i,k)\in \mathcal{C}(\omega,i,k) \\
r(\omega,i,k)\neq r_\t(\omega,i,k)}}
\frac{{L}_{i,k}^{\omega}-r(\omega,i,k)}{r_{\t}(\omega,i,k)-r(\omega,i,k)}\in \HH_{n,\K} .
$$
Note that $L_{i,\ell}^\omega=L_i$. The element $\widetilde{F}_\t$ clearly contains $$
F_\t=\prod\limits^n\limits_{i=1}\prod\limits_{\substack{c\in C(i)\\c\neq c_{\t}(i)}}\frac{L_i-c}{c_{\t}(i)-c}
$$
as a factor. More precisely, we have that
$$
\widetilde{F}_\t = F_\t \cdot \prod_{i =1}^{n}\frac{L_i}{c_{\t}(i)}.
$$
Applying \cite[Proposition 2.6, Theorem 2.15]{Ma} and \cite[Proposition 3.4, Theorem 3.16]{MathasSeminormal}, we see that each of the above element acts on $F_\t$ as the identity. This completes the proof of the lemma.
\end{proof}

The dominance order ``$\unrhd$'' can be naturally generalized from the set of standard tableaux to the set of semistandard tableaux, see \cite[Definition 3.6]{JM}.

\begin{dfn} Let $\bS\in\mT(\blam,\bmu)$, $\bT\in\mT(\blam,\bnu)$, $\bU\in\mT(\brho,\balpha)$ and $\bV\in\mT(\brho,\bbeta)$. If either $\brho\rhd\blam$, or $\blam=\brho$ and $\bU\unrhd\bS$ and $\bV\unrhd\bT$, then we write $$
(\bU,\bV)\unrhd (\bS,\bT) .
$$
If $(\bU,\bV)\unrhd (\bS,\bT)$ and $(\bU,\bV)\neq(\bS,\bT)$, then we write $(\bU,\bV)\rhd (\bS,\bT)$.
\end{dfn}

By the general theory (\cite{MathasSeminormal}) of cellular algebras with a family of JM elements which separate $\cup_{\substack{\blam\in\mC^+\\\ \bmu\in\mC}}\mT(\blam,\bmu)$ over $\K$, we get the following result.

\begin{prop}\text{(\cite[Lemma 3.3, Theorems 3.7, 3.16]{MathasSeminormal}, \cite[Theorem 3.8]{GL})}\label{Fseminormal} 1) For any $\bS\in\mT(\blam,\bmu)$ and $\bT\in\mT(\blam,\bnu)$, we have $$\begin{aligned}
\varphi_{\bS\bT}^{\O_K}&=\mathcal{F}_{\bS\bT}+\sum_{\substack{\bU\in\mT(\brho,\bmu),\bV\in\mT(\brho,\bnu)\\
(\bU,\bV)\rhd (\bS,\bT)}}a_{\bU\bV}^{\bS\bT}\mathcal{F}_{\bU\bV},\\
\mathcal{F}_{\bS\bT}&=\varphi_{\bS\bT}^{\O_K}+\sum_{\substack{\bU\in\mT(\brho,\bmu),\bV\in\mT(\brho,\bnu)\\(\bU,\bV)\rhd (\bS,\bT)}}b_{\bU\bV}^{\bS\bT}\varphi_{\bU\bV}^{\O_K},
\end{aligned}
$$
where $a_{\bU\bV}^{\bS\bT}, b_{\bU\bV}^{\bS\bT}\in\K$ for each pair $(\bU,\bV)$.

2) $\mathcal{F}_{\bS\bT}\mathcal{F}_{\bU\bV}=\delta_{\bT\bU}\gamma_{\bT}\mathcal{F}_{\bS\bV}$ for some $\gamma_{\bT}\in\K^\times$. Moreover, $\gamma_{\bT^\blam}=1$;

3) The set \begin{equation}\label{FseminormaBasis}
\bigl\{\mathcal{F}_{\bS\bT}\bigm|\bS\in\mT(\blam,\bmu), \bT\in\mT(\blam,\bnu),\blam\in\mC^+,\bmu,\bnu\in\mC\bigr\}
\end{equation}
forms a $\K$-basis of the semisimple $\K$-algebra $S_{n,\K}$;

4) For each $\bS\in\mT(\blam,\bmu)$, $\mathcal{F}_{\bS}=\mathcal{F}_{\bS\bS}/\gamma_{\bS}$ and $S_{n,\K} \mathcal{F}_{\bS}\cong\Delta_{\K}^\blam$. The set $$\bigl\{\mathcal{F}_{\bS\bS}/\gamma_{\bS}\bigm|\bS\in\mT(\blam,\bmu), \blam\in\mC^+,\bmu\in\mC\bigr\}
$$
forms a complete set of pairwise orthogonal primitive idempotents of $S_{n,\K}$.
\end{prop}

\begin{cor}\label{Tlamcor} Let $\blam\in\mC^+$. Then $$
\varphi_{\bT^\blam\bT^\blam}^{\O_K}=\mathcal{F}_{\bT^\blam\bT^\blam}+\sum_{\substack{\bS\in\mT(\brho,\blam),\bT\in\mT(\brho,\blam)\\ \blam\lhd\brho\in\mC^+}}a_{\bS\bT}^{\blam}\mathcal{F}_{\bS\bT},
$$
where $a_{\bS\bT}^{\blam}\in\K$ for each pair $(\bS,\bT)$.
\end{cor}

\begin{prop}\label{formextend1} Assume that $\omega:=(\emptyset,\cdots,\emptyset,(1^n))\in\mC^+$. Then the symmetrizing form $\tau_{\K}$ on $\HH_{n,\K}$ can be uniquely extended to a symmetrizing form $\widetilde{\tau}_{\K}$ on $S_{n,\K}$ such that $$
\widetilde{\tau}_{\K}(\mathcal{F}_{\bS\bT})=\delta_{\bS,\bT}\frac{\gamma_{\bS}}{s_\blam},\quad\forall\,\bS\in\mathcal{T}_0(\blam,\bmu),\bT\in\mathcal{T}_0(\blam,\bnu), \blam\in\Parts[\ell,n], \bmu,\bnu\in\mC,
$$
where $s_\blam$ is the Schur element of the cyclotomic Hecke algebra $\HH_{n,\K}$ corresponding to $\blam\in\Parts[\ell,n]$, see \cite[Definition 1.5]{Ma} and \cite{Zh}. Furthermore, for any $x\in S_{n,\K}$, there exists an element $h\in\HH_{n,\K}$, such that \begin{equation}\label{equiv1}
x\equiv h\,\,\,\pmod{\bigl[S_{n,\K} ,S_{n,\K}\bigr]} .
\end{equation}
\end{prop}

\begin{proof} For any $\bS\in\mathcal{T}_0(\blam,\bmu),\bT\in\mathcal{T}_0(\blam,\bnu), \bU\in\mathcal{T}_0(\brho,\balpha),\bV\in\mathcal{T}_0(\brho,\bbeta), \blam,\brho\in\Parts[\ell,n], \bmu,\bnu,\balpha,\bbeta\in\mC$, we have $$\begin{aligned}
\widetilde{\tau}_{\K}(\mathcal{F}_{\bS\bT}\mathcal{F}_{\bU\bV})&=\widetilde{\tau}_{\K}(\delta_{\bT,\bU}\gamma_{\bT}\mathcal{F}_{\bS\bV})=\delta_{\bT,\bU}\delta_{\bS,\bV}\gamma_{\bT}\frac{\gamma_{\bS}}{s_\blam}\\
&=\delta_{\bS,\bV}\gamma_{\bS}\widetilde{\tau}_{\K}(\mathcal{F}_{\bT\bU})=\widetilde{\tau}_{\K}(\mathcal{F}_{\bU\bV}\mathcal{F}_{\bS\bT}).
\end{aligned}
$$
This proves that $\widetilde{\tau}_{\K}$ is a symmetric form on $S_{n,\K}$. Applying Proposition \ref{Fseminormal}, it is easy to see that $\widetilde{\tau}_{\K}$ is non-degenerate and hence a symmetrizing form on $S_{n,\K}$. Using Lemma \ref{Ftsemi2} and \cite[Lemma 1.6]{Ma} we also see that $\widetilde{\tau}_{\K}$ is an extension of $\tau_\K$ from $\HH_{n,\K}$ to $S_{n,\K}$. This proves the first part of the proposition.

For the second part, we note that $$
\mathcal{F}_{\bS\bT}=\mathcal{F}_{\bS\bT}\mathcal{F}_{\bT}/\gamma_{\bT}-0=\mathcal{F}_{\bS\bT}\mathcal{F}_{\bT}/\gamma_{\bT}-
(\mathcal{F}_{\bT}/\gamma_{\bT})\mathcal{F}_{\bS\bT}\in\bigl[S_{n,\K} ,S_{n,\K}\bigr],
$$ whenever $\bS\neq\bT$. Thus the second part  of the proposition follows from the equality $$
\mathcal{F}_{\bT\bT}=\frac{1}{\gamma_{\tlam}}\mathcal{F}_{\bT\tlam}\mathcal{F}_{\tlam\bT}\equiv \frac{1}{\gamma_{\tlam}}\mathcal{F}_{\tlam\bT}\mathcal{F}_{\bT\tlam}
=\frac{\gamma_{\bT}}{\gamma_{\tlam}}f_{\tlam\tlam}\,\,\,\pmod{\bigl[S_{n,\K},S_{n,\K}\bigr]} .
$$
This completes the proof of the proposition.
\end{proof}

\begin{lem}\label{fieldextension} Let $A$ be a finite dimensional algebra over a field $F$ and $L$ is an extension of $F$. Then there are canonical isomorphism: $$
L\otimes_{F}Z(A)\cong Z(L\otimes_{F}A) .
$$
In particular, $\dim_F Z(A)=\dim_L Z(L\otimes_{F}A)$. Moreover, $$L\otimes_F A/[A,A]\cong (L\otimes_{F}A)/[L\otimes_{F}A,L\otimes_{F}A], $$ and hence
$\dim_F A/[A,A]=\dim_L (L\otimes_{F}A)/[L\otimes_{F}A,L\otimes_{F}A]$.
\end{lem}

\begin{proof} It is clear that the canonical morphism $L\otimes_{F}Z(A)\rightarrow Z(L\otimes_{F}A)$ is an injection. Note that $\dim_L Z(L\otimes_{F}A)$ is the same as
the $L$-rank of the solution space of some related system of linear equations (which is defined over $F$). It follows that $\dim_F Z(A)=\dim_K Z(L\otimes_{F}A)$ and hence the previous canonical morphism is an isomorphism.
This proves the first isomorphism of the lemma.

The second isomorphism follows from the fact that $L\otimes_F [A,A]=[L\otimes_F A, L\otimes_F A]$.
\end{proof}

The next lemma gives some bases results on the center and cocenter of the semisimple cyclotomic Schur algebra $S_{n,\K}$, as well as some dimension bound results on the center and cocenter of the cyclotomic Schur algebra $S_{n,K}$ over an arbitrary field $K$.

\begin{lem}\label{Kbasis0} 1) The set \begin{equation}\label{Fblam}\biggl\{\mathcal{F}_\blam:=\sum_{\substack{\bmu\in\mC\\ \bS\in\mT(\blam,\bmu)}}\mathcal{F}_{\bS\bS}/\gamma_{\bS}\biggm|\blam\in\mC^+\biggr\}\end{equation}
is a $\K$-basis of $Z(S_{n,\K} )$;

2) The set $$\bigl\{\mathcal{F}_{\bT^\blam\bT^\blam}+[S_{n,\K} ,S_{n,\K} ]\bigm|\blam\in\mC^+\bigr\}
$$ is a $\K$-basis of the cocenter $\Tr(S_{n,\K}):=S_{n,\K} /[S_{n,\K} ,S_{n,\K} ]$;

3) We have $$
\dim_K Z(S_{n,K} )\geq \#\mC^+ =\dim_{\K} Z(S_{n,\K} ),\quad\,
\dim_K [S_{n,K} ,S_{n,K} ]\leq\dim_{\K}[S_{n,\K} ,S_{n,\K} ].
$$
In particular, $\dim\Tr(S_{n,K})=\dim S_{n,K}/[S_{n,K} ,S_{n,K}])\geq\#\mC^+$.
\end{lem}

\begin{proof} Since $S_{n,\K} $ is a split semisimple $\K$-algebra and has a seminormal basis $$\{\mathcal{F}_{\bS\bT}|\bS\in\mT(\blam,\bmu),\bT\in\mT(\blam,\bnu),\blam\in\mC^+, \bmu,\bnu\in\mC\}, $$ we have $$\begin{aligned}
& \mathcal{F}_{\bS\bT}(\mathcal{F}_{\bT\bT}/\gamma_\bT)-(\mathcal{F}_{\bT\bT}/\gamma_\bT) \mathcal{F}_{\bS\bT}=\mathcal{F}_{\bS\bT}-\delta_{\bS,\bT}\mathcal{F}_{\bT\bT},\\
& (\mathcal{F}_{\bT\bT}/\gamma_\bT)-(\mathcal{F}_{\bT^\blam\bT^\blam}/\gamma_{\bT^\blam})=(\mathcal{F}_{\bT\bT^\blam}/\gamma_\bT)(\mathcal{F}_{\bT^\blam\bT}/\gamma_{\bT^\blam})
-(\mathcal{F}_{\bT^\blam\bT}/\gamma_{\bT^\blam})(\mathcal{F}_{\bT\bT^\blam})/\gamma_{\bT}).
\end{aligned}
$$
It follows that (\ref{Fblam}) is a $\K$-basis of $Z(S_{n,\K} )$, and \begin{equation}\label{Kcommutator0}
\Biggl\{\mathcal{F}_{\bS\bT},(\mathcal{F}_{\bU\bU}/\gamma_\bU)-(\mathcal{F}_{\bT^\blam\bT^\blam}/\gamma_{\bT^\blam})\Biggm|\begin{matrix}\text{$\bU\in\mT(\blam,\bmu), \bS\in\mT(\brho,\balpha),\bT\in\mT(\brho,\bbeta),\bU\neq \bT^\blam, \bS\neq\bT$},\\
\text{$\blam,\brho\in\mC^+, \bmu,\balpha,\bbeta\in\mC$}\end{matrix}\Biggr\}
\end{equation}
is a $\K$-basis of $[S_{n,\K} ,S_{n,\K} ]$ and part 2 follows. In particular, $\dim_{\K} Z(S_{n,\K} )=\#\mC^+$.

Note that $K\otimes_{{\O_K}}S_{n,{\O_K}} \cong S_{n,K} $. We have $\dim_K Z(S_{n,K} )$ is the same as the $K$-rank of the solution space of some related system of linear equations (which is defined over ${\O_K}\subset\K$). It follows that $\dim_K Z(S_{n,K} )\geq\dim_{\K} Z(S_{n,\K} )=\#\mC^+$ as required. The final inequality follows from following inequality: $$
\dim_{K}[S_{n,K} ,S_{n,K} ]\leq\rank_{{\O_K}}[S_{n,{\O_K}} ,S_{n,{\O_K}} ]=\dim_{\K}[S_{n,\K} ,S_{n,\K} ].
$$
This completes the proof of the lemma.
\end{proof}

The following theorem is the third main result of this paper.

\begin{thm}\label{mainthm33} Let $\O\in\{K,\O_K,\K\}$. The cocenter $S_{n,\O} /[S_{n,\O} ,S_{n,\O} ]$ of $S_{n,\O} $ is a free $\O$-module and the following set \begin{equation}\label{Ogenerators}
\bigl\{\varphi_{\bT^\blam\bT^\blam}^{\O}+[S_{n,\O} ,S_{n,\O} ]\bigm|\blam\in\mC^+\bigr\}
\end{equation}
forms an $\O$-basis of the cocenter $\Tr(S_{n,\O})=S_{n,\O}/[S_{n,\O} ,S_{n,\O} ]$. In particular, $$
\dim_K \Tr(S_{n,K})=\dim S_{n,K} /[S_{n,K} ,S_{n,K} ]=\#\mC^+ .
$$
\end{thm}

\begin{proof} For any $\blam\in\mC^+, \bmu,\bnu\in\mC, \bS\in\mathcal{T}_0(\blam,\bmu), \bT\in\mathcal{T}_0(\blam,\bnu)$, we have (as $\<\varphi^{\O_K}_{\bT^\blam},\varphi^{\O_K}_{\bT^\blam}\>=1$ by \eqref{bilinear2}) that $$
\varphi^{\O_K}_{\bS\bT}\equiv\varphi^{\O_K}_{\bS\bT^\blam}\varphi^{\O_K}_{\bT^\blam\bT}\,\,\pmod{\bigl(S_{n,{\O_K}} \bigr)^{\rhd\blam}},
$$
It follows that $$
\varphi^{\O_K}_{\bS\bT}\equiv\varphi^{\O_K}_{\bS\bT^\blam}\varphi_{\bT^\blam\bT}^{\O_K}\equiv \varphi^{\O_K}_{\bT^\blam\bT}\varphi^{\O_K}_{\bS\bT^\blam}
\equiv\<\varphi^{\O_K}_{\bT},\varphi^{\O_K}_{\bS}\>\varphi^{\O_K}_{\bT^\blam\bT^\blam}\,\,\pmod{\bigl(S_{n,{\O_K}} \bigr)^{\rhd\blam}+[S_{n,{\O_K}},S_{n,{\O_K}}]}.
$$
By a downward induction on dominance order, we can deduce that \begin{equation}\label{expansion1}
\varphi^{\O_K}_{\bS\bT}\equiv\sum_{\brho\unrhd\blam}a_{\bS,\bT}^{\brho}\varphi^{\O_K}_{\bT^\brho\bT^\brho}\pmod{[S_{n,{\O_K}},S_{n,{\O_K}}]},
\end{equation}
where $a_{\bS,\bT}^{\brho}\in{\O_K}$ for each $\brho$. This proves that (\ref{Ogenerators}) is a set of ${\O_K}$-linear generator of the cocenter $\Tr(S_{n,\O_K})=S_{n,{\O_K}} /[S_{n,{\O_K}} ,S_{n,{\O_K}} ]$. As a result, \begin{equation}\label{Ogenerators2}
\bigl\{\varphi_{\bT^\blam\bT^\blam}^K+[S_{n,K} ,S_{n,K} ]\bigm|\blam\in\mC^+\bigr\}
\end{equation}
is a set of $K$-linear generator of the cocenter $\Tr(S_{n,K})=S_{n,K} /[S_{n,K} ,S_{n,K} ]$.

On the other hand, by Lemma \ref{Kbasis0} 3), we know that
$\dim_K\Tr(S_{n,K})=\dim_K S_{n,K}/[S_{n,K} ,S_{n,K} ]\geq\#\mC^+$. Thus (\ref{Ogenerators2}) must be a $K$-basis of the cocenter $\Tr(S_{n,K})$.
Furthermore, we can deduce that (\ref{Ogenerators}) (as the preimage of (\ref{Ogenerators2})) must be ${\O_K}$-linear independent and hence it forms an ${\O_K}$-basis of the cocenter $\Tr(S_{n,\O_K})=S_{n,{\O_K}} /[S_{n,{\O_K}} ,S_{n,{\O_K}} ]$.

The remaining part of the theorem follows from the natural isomorphism $\K\otimes_{\O_K}S_{n,\O_K} /[S_{n,\O_K} ,S_{n,\O_K} ]\cong S_{n,\K} /[S_{n,\K} ,S_{n,\K} ]$. This completes the proof of the theorem.
\end{proof}

\begin{cor}\label{maincor1} Both the canonical map $K\otimes_{{\O_K}}[S_{n,{\O_K}} ,S_{n,{\O_K}} ]\rightarrow [S_{n,K} ,S_{n,K} ]$ and the canonical map
$K\otimes_{{\O_K}}S_{n,{\O_K}} /[S_{n,{\O_K}} ,S_{n,{\O_K}} ]\rightarrow S_{n,K} /[S_{n,K} ,S_{n,K} ]$ are isomorphisms. In particular, $[S_{n,{\O_K}} ,S_{n,{\O_K}} ]$ is a pure ${\O_K}$-submodule of $S_{n,{\O_K}} $.
\end{cor}

\begin{proof} Since the canonical map $K\otimes_{{\O_K}}[S_{n,{\O_K}},S_{n,{\O_K}}] \to K\otimes_{{\O_K}}S_{n,{\O_K}} \xrightarrow{\sim} S_{n,K} $ factors through $[S_{n,K},S_{n,K}]$, it follows that the canonical map
$K\otimes_{{\O_K}}S_{n,{\O_K}} /[S_{n,{\O_K}} ,S_{n,{\O_K}} ]\rightarrow S_{n,K} /[S_{n,K} ,S_{n,K} ]$ is surjective. Now using Theorem \ref{mainthm33} and comparing dimensions we see that this surjection is an isomorphism.

By Theorem \ref{mainthm33}, $S_{n,{\O_K}} /[S_{n,{\O_K}} ,S_{n,{\O_K}} ]$ is a free ${\O_K}$-module which implies that $[S_{n,{\O_K}} ,S_{n,{\O_K}} ]$ is a pure ${\O_K}$-submodule of $S_{n,{\O_K}} $. This in turn implies that the canonical map $K\otimes_{{\O_K}}[S_{n,{\O_K}} ,S_{n,{\O_K}} ]\rightarrow [S_{n,K} ,S_{n,K} ]$ is an injection. Finally, we see this injection is an isomorphism by comparing their dimensions on both sides.
\end{proof}


Let $\O\in\{K,\O_K,\K\}$. For any $\bS\in\mathcal{T}_0(\blam,\bmu),\bT\in\mathcal{T}_0(\blam,\bnu)$ with $(\bS,\bT)\neq(\bT^\blam,\bT^\blam)$,  there exist (by (\ref{expansion1})) unique scalars $\{a_{\bS\bT}^\bbeta\in\O|\blam\unlhd\bbeta\in\mC^+\}$, such that  $$
\varphi^\O_{\bS\bT}\equiv \sum_{\bbeta\unrhd\blam}a_{\bS\bT}^\bbeta\varphi_{\bT^\bbeta\bT^\bbeta}^\O\,\,\pmod{[S_{n,\O} ,S_{n,\O} ]} . $$
Moreover, it is easy to see that $$
a_{\bS\bT}^\blam=\<\varphi^\O_\bS,\varphi^\O_\bT\> .
$$

\begin{cor}\label{keycor3} Let $\O\in\{K,\O_K,\K\}$. With the notations as above, the free $\O$-module $[S_{n,\O} ,S_{n,\O} ]$ has an $\O$-basis of the form \begin{equation}\label{commuspan}
\Biggl\{\varphi^\O_{\bS\bT}-\sum_{\brho\unrhd\blam}a_{\bS\bT}^\brho\varphi^\O_{\bT^\brho\bT^\brho}\Biggm|\begin{matrix}\text{$\bS\in\mathcal{T}_0(\blam,\bmu),
\bT\in\mathcal{T}_0(\blam,\bnu), (\bS,\bT)\neq(\bT^\blam,\bT^\blam)$}\\  \text{$\blam\in\mC^+, \bmu,\bnu\in\mC$}
\end{matrix}\Biggr\}.
\end{equation}
\end{cor}

\begin{proof} Assume $\O\in\{K,\K\}$. Using Theorem \ref{mainthm33}, it is  easy to see that (\ref{commuspan}) is a subset of $[S_{n,\O} ,S_{n,\O} ]$ and its elements are $\O$-linearly independent. Hence they must form a basis by dimensions comparison and Corollary \ref{maincor1}.

As a consequence, we can use Nakayama's lemma to deduce that the elements in (\ref{commuspan}) form an $\O_K$-basis of $[S_{n,\O_K} ,S_{n,\O_K} ]$ when $\O=\O_K$.
\end{proof}

\end{document}